%% file: arxiv.tex
\documentclass{article}

\usepackage[preprint,nonatbib]{neurips_2020}
\usepackage[numbers]{natbib}

\usepackage{xcolor}
\definecolor{lightblue}{rgb}{0.63, 0.74, 0.78}
\definecolor{orange}{rgb}{0.85, 0.55, 0.13}
\definecolor{silver}{rgb}{0.69, 0.67, 0.66}
\definecolor{rust}{rgb}{0.72, 0.26, 0.06}

\colorlet{lightsilver}{silver!30!white}
\colorlet{darkorange}{orange!85!black}
\colorlet{darksilver}{silver!85!black}
\colorlet{darklightblue}{lightblue!85!black}
\colorlet{darkrust}{rust!85!black}

\usepackage[utf8]{inputenc} 
\usepackage[T1]{fontenc}    
\usepackage{hyperref}
\hypersetup{colorlinks=true,linkcolor=darkrust,citecolor=darklightblue,urlcolor=darksilver}       
\usepackage{url}            
\usepackage{booktabs}       
\usepackage{amsfonts}       
\usepackage{nicefrac}       
\usepackage{microtype}      

\usepackage{amsmath}
\usepackage{amssymb}
\usepackage{amsthm}
\usepackage{mathtools}
\usepackage{colortbl}
\usepackage{algorithm}
\usepackage{algorithmic}
\usepackage{tikz}
\usepackage{wrapfig}
\usepackage{enumitem}
\usepackage{cite}

\theoremstyle{definition}
\newtheorem{definition}{Definition}

\theoremstyle{plain}
\newtheorem{theorem}{Theorem}
\newtheorem{lemma}[theorem]{Lemma}

\newcommand{\Reals}{\mathbb{R}}
\newcommand{\Herms}{\mathbb{H}}
\newcommand{\Id}{\operatorname{Id}}
\newcommand{\defeq}{\coloneqq}

\newcommand{\C}{\mathcal{C}}
\newcommand{\tC}{\tilde{\mathcal{C}}}
\newcommand{\K}{\mathcal{K}}
\newcommand{\tK}{\tilde{\mathcal{K}}}

\newcommand{\Exp}{\operatorname{Exp}}

\newcommand{\interior}{\operatorname{int}}
\newcommand*{\dr}{\,\mathrm{d}r}
\newcommand*{\diff}{\,\mathrm{d}}
\newcommand{\tf}{\tilde{f}}
\newcommand{\tg}{\tilde{g}}
\newcommand{\tm}{\tilde{m}}
\newcommand{\tn}{\tilde{n}}
\newcommand{\tr}{\operatorname{tr}}
\newcommand{\logdet}{\operatorname{logdet}}
\newcommand{\M}{\mathcal{M}}
\newcommand{\T}{\mathcal{T}}
\newcommand{\diag}{\operatorname{diag}}

\DeclareMathOperator*{\argmin}{argmin}

\DeclarePairedDelimiterX{\infdivx}[2]{(}{)}{%
  #1\;\delimsize\|\;#2%
}
\newcommand{\Div}[1]{\mathbb{D}_{#1}\infdivx}

\title{Competitive Mirror Descent}

\author{%
  Florian Sch{\"a}fer\\
  Caltech \\
  \texttt{schaefer@caltech.edu} 
  \And
  Anima Anandkumar \\
  Caltech \\
  \texttt{anima@caltech.edu}
  \And
  Houman Owhadi \\
  Caltech \\
  \texttt{owhadi@caltech.edu}
}

\begin{document}

\maketitle

\begin{abstract}
  Constrained competitive optimization involves multiple agents trying to minimize conflicting objectives, subject to constraints. This is a highly expressive modeling language that subsumes most of modern machine learning.
  In this work we propose competitive mirror descent (CMD): a general method for solving such problems based on first order information that can be obtained by automatic differentiation.
  First, by adding Lagrange multipliers, we obtain a simplified constraint set with an associated Bregman potential.
  At each iteration, we then solve for the Nash equilibrium of a regularized bilinear approximation of the full problem to obtain a direction of movement of the agents.
  Finally, we obtain the next iterate by following this direction according to the dual geometry induced by the Bregman potential.
  By using the dual geometry we obtain feasible iterates despite only solving a linear system at each iteration, eliminating the need for projection steps while still accounting for the global nonlinear structure of the constraint set.
  As a special case we obtain a novel competitive multiplicative weights algorithm for problems on the positive cone.
\end{abstract}

\section{Introduction}

\paragraph{Constrained competitive optimization:} Machine learning replaces domain-specific modeling by using training data to pick the best performing procedure from a large class of candidates.
This is usually done by solving a large unconstrained optimization problem using first order information obtained from automatic differentiation, which we refer to as \emph{training}.\\
Recently, there has been a surge in attempts at introducing structure into machine learning models in order to improve accuracy, reliability, and robustness.
This structure may come from physical laws \citep{cohen2016group,stewart2017label,raissi2019physics,anderson2019cormorant}, safety requirements or optimality conditions in reinforcement learning \citep{achiam2017constrained,miryoosefi2019reinforcement,bacon2019lagrangian}, or fairness requirements \citep{cotter2018training,cotter2018two,narasimhan2019optimizing}.
Imposing constraints during training provides a flexible and modular framework to incorporate these requirements.\\
Another active area of research is to replace the training problem by a competitive optimization problem where two agents are optimizing their own objective, in competition with each other. 
This approach has been successfully applied to adversarial robustness \citep{madry2017towards} and generative modeling \citep{goodfellow2014generative}.\\
Together, this yields constrained competitive optimization of the general form
\begin{equation}
  \label{eqn:intro-cco}
  \min_{\substack{x \in \C, \\ \tf(x) \in \tC}} f(x, y), \quad \min_{\substack{y \in \K, \\ \tg(y) \in \tK}} g(x, y).
\end{equation}
The goal of this work is to provide a general purpose algorithm for solving constrained competitive problems, as a counterpart of gradient descent in unconstrained single-agent optimization.

\paragraph{Constraints and competition:} 
At a closer look, constraints and competition are closely related. 
In the special case of \emph{equality} constraints we can use a Lagrange multiplier $\lambda$ to turn a constrained minimization problem into an unconstrained minimax problem
\begin{equation*}
  \min \limits_{x:\ g(x) = 0} f(x) \quad \Leftrightarrow \quad \min \limits_{x} \max_{\lambda} f(x) + \lambda g(x).
\end{equation*}
By simultaneously optimizing over $x$ and $\lambda$ we obtain an unconstrained competitive optimization problem that approximates the original constrained problem and, under certain conditions, recovers it.
Thus, any method for competitive optimization can immediately be applied to equality constrained competitive problems.

\paragraph{Competitive gradient descent (CGD):} 
The simplest approach to solving a minmax problem is simultaneous gradient descent (SimGD), where both players perform a step of gradient descent (ascent for the dual player) at each iteration.
However, this algorithm has poor performance in practice and diverges even on simple problems.
The authors of \citep{schafer2019competitive} argue that the poor performance of SimGD comes from the fact that the updates of the two players are computed without interactive nature of the game into account.
They propose to instead compute the updates of both players as the Nash equilibrium of a local bilinear approximation of the original game, regularized with a quadratic penalty that models the limited confidence of both players in the approximation's accuracy.
The resulting algorithm, \emph{competitive gradient descent (CGD)}, is shown to have improved convergence properties, making it a promising candidate for solving constrained competitive optimization problems.

\paragraph{CGD and inequality constraints:} 
While Lagrange multipliers can eliminate equality constraints, they can merely simplify inequality constraints, as
\begin{equation*}
  \min \limits_{x:\ h(x) \leq 0} f(x) \quad \Leftrightarrow \quad \min \limits_{x} \max_{\mu \geq 0} f(x) + \mu h(x).
\end{equation*}
In general, Lagrangian duality can replace a nonlinear conic constraint $h(x) \in \C$ with a linear constraint on the polar cone $\lambda \in \C^{\circ} \defeq \{y : \sup_{x \in \C} x^{\top} y \leq 0 \}$.
Thus, applying CGD to conically constrained optimization requires a version of CGD that can handle linear conic constraints.
The simplest approach to extend CGD to conic constraints is to interleave its updates with projections onto the constraint set, obtaining projected gradient descent (PCGD).
But this algorithm converges to suboptimal points even in convex examples due to a phenomenon that we call \emph{empty threats}.
Empty threats arise when a player can improve its loss in the local approximation by violating the constraints. 
Anticipating this action, which is ultimately prevented by the projection step, the other player deviates from the optimal strategy, preventing PCGD from converging to the optimal solution (c.f. Figure~\ref{fig:empty-threats-intro}). 

\begin{figure}
  \centering
\begin{minipage}{0.32\textwidth}
  \includegraphics[width=1.00\textwidth]{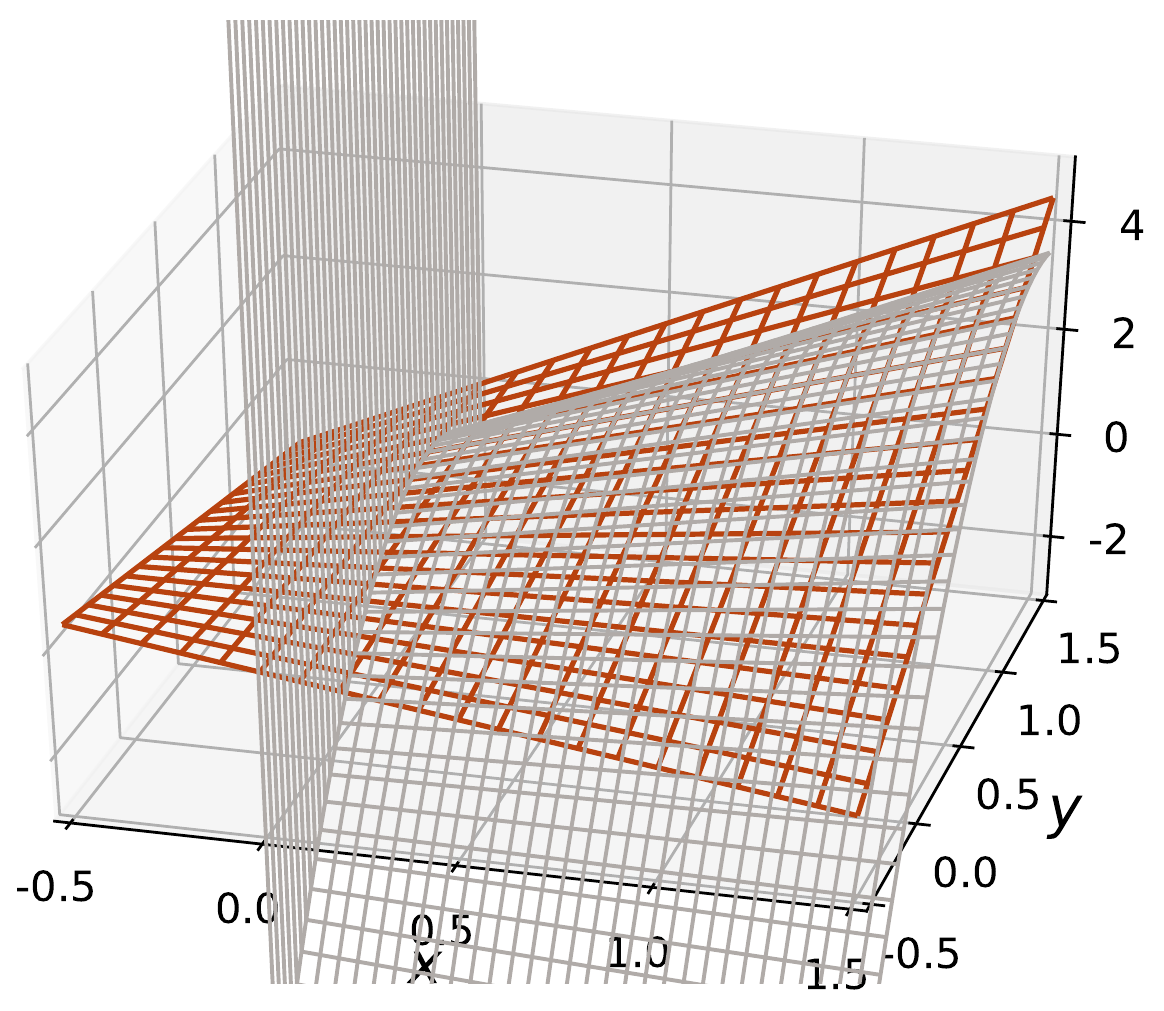}
  \caption{\label{fig:empty-threats-intro} The local approximation ({\color{darkrust} coarse mesh}) of the constrained minimax problem ({\color{darksilver} fine mesh}) is unconstrained.
  Thus, $y$ chooses its move assuming $x$ to turn negative. After $x$ is projected back to the feasible set, this move of $y$ is suboptimal.}
\end{minipage} \hfill
\begin{minipage}{0.32\textwidth}
  \includegraphics[width=1.00\textwidth]{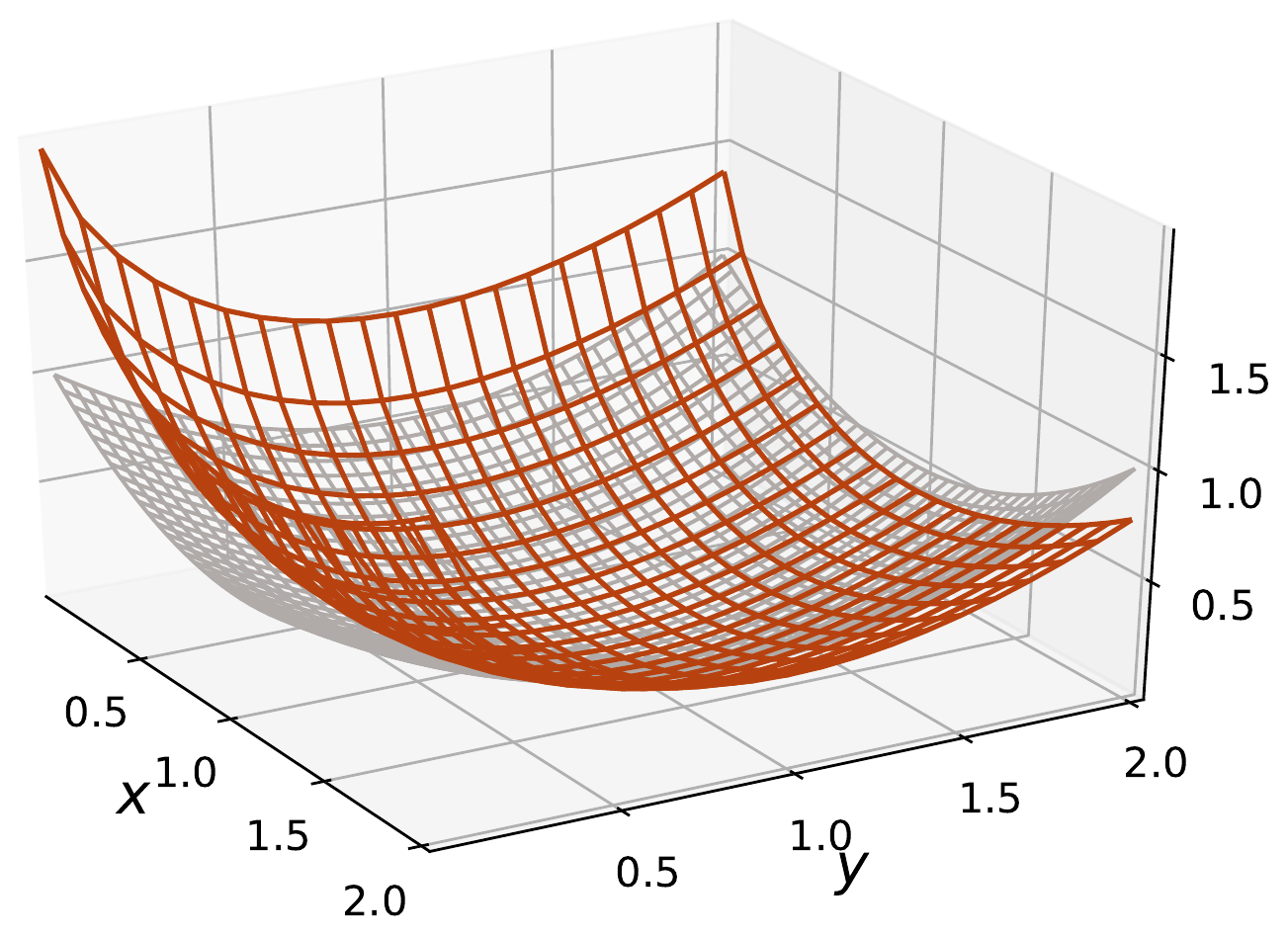}
  \caption{\label{fig:divergences-intro} Squared distance ({\color{darksilver} fine mesh}) and KL divergence ({\color{darkrust} coarse mesh}) to the point $(1,1)$. The KL divergence increases sharply as $(x,y)$ approaches the boundary of the positive orthant.}
\end{minipage} \hfill
\begin{minipage}{0.32\textwidth}
  \includegraphics[width=1.00\textwidth]{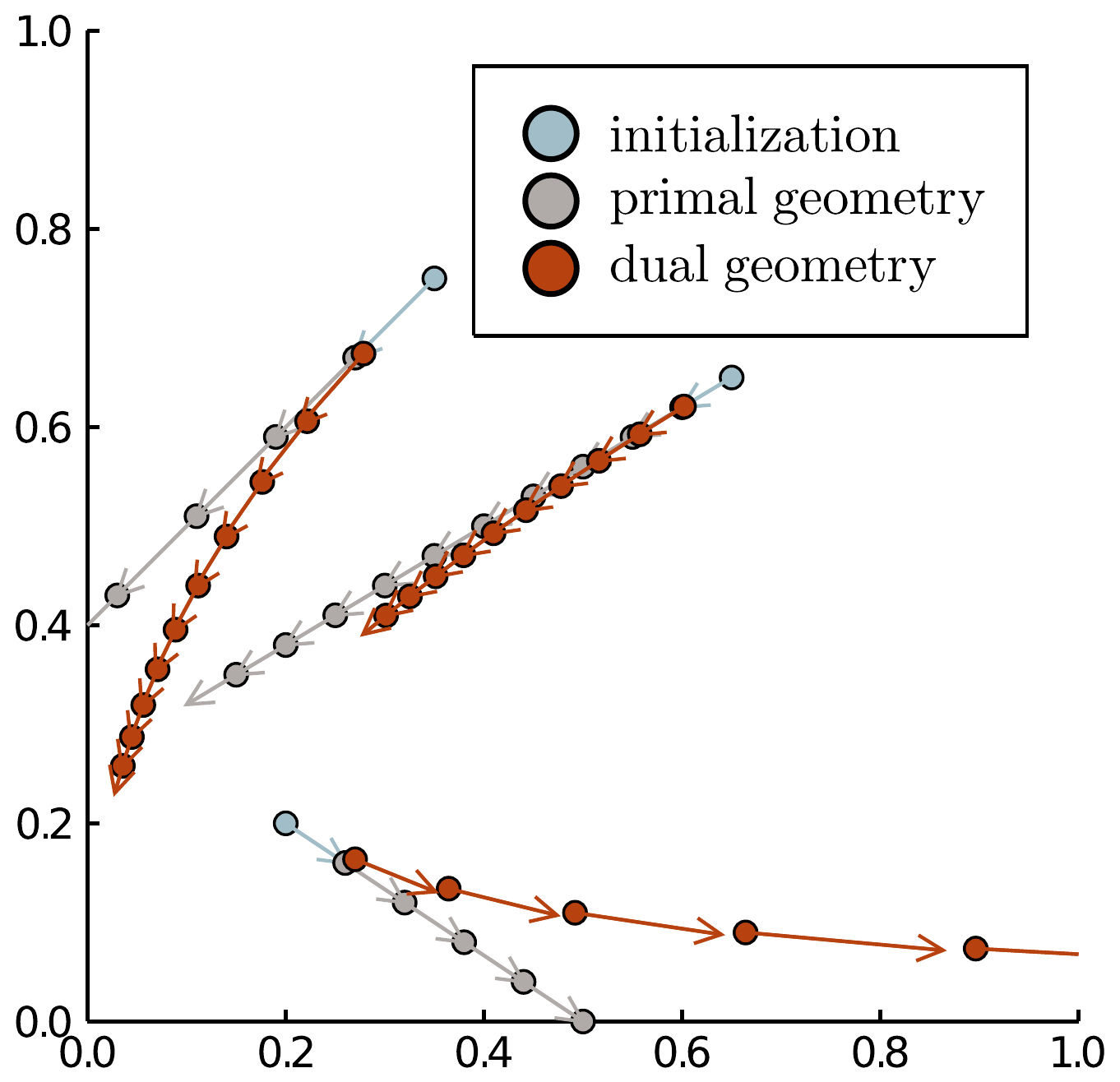}
  \caption{\label{fig:dual-geometry-intro} The dual notion of straight line induced by the Shannon entropy guarantees feasible iterates on $\Reals_+^m$.}
\end{minipage}
\end{figure}

\paragraph{Mirror descent:} 
The fundamental reason for the \emph{empty threats}-phenomenon observed in projected CGD is that the local update rule of CGD does not include any information about the global structure of the constraint set.
The \emph{mirror descent} framework \citep{nemirovsky1983problem} uses a \emph{Bregman potential} $\psi$ to obtain a local update rule that takes the global geometry of the constraint set into account.
It computes the next iterate $x_{k+1}$ as the minimizer of a linear approximation of the objective, regularized with the \emph{Bregman divergence} $\Div{\psi}{x_{k + 1}}{x_{k}}$ associated to $\psi$.
For problems on the positive orthant $\Reals_+^m$, we can ensure feasible iterates by choosing $\psi$ as the Shannon entropy, obtaining an algorithm known as \emph{entropic mirror descent}, \emph{exponentiated gradient descent}, or \emph{multiplicative weights algorithm}.
A naive extension of mirror descent to the competitive setting simply replaces the quadratic regularization in SimGD or CGD with a Bregman divergence.
However, the former inherits the cycling problem of SimGD, while the later requires solving a nonlinear problem \emph{at each iteration}.

\begin{algorithm}[H]
    \begin{algorithmic}[1]
        \FOR{$0 \leq k < N$}
           \STATE $\Delta x = -\left(\diag_{x_{k}} - \left[D_{xy}^2 f\right] \diag_{y_k}^{-1} \left[D_{yx}^2 g\right] \right)^{-1} \left(  \left[\nabla_x f \right] - \left[D_{xy}^2 f \right] \diag_{y_{k}}^{-1} \left[\nabla_y g\right] \right)$
           \STATE $\Delta y = -\left(\diag_{y_{k}} - \left[D_{yx}^2 g\right] \diag_{x_{k}}^{-1} \left[D_{xy}^2 f\right] \right)^{-1} \left( \left[\nabla_x g \right] - \left[D_{yx}^2 g \right] \diag_{x_{k}}^{-1} \left[\nabla_x f\right]  \right)$
           \STATE $x_{k + 1} = x_{k} \exp\left(\Delta x\right)$
           \STATE $y_{k + 1} = y_{k} \exp\left(\Delta y\right)$
  			\ENDFOR
  			\RETURN $x_{N}, y_{N}$
	\end{algorithmic}
	\caption{\label{alg:cmw} Competitive multiplicative weights (CMW)\protect\footnotemark}
\end{algorithm}

\paragraph{Our contributions:}
In this work, we propose competitive mirror descent (CMD), a novel algorithm that combines the ideas of CMD and mirror descent in a computationally efficient way.
In the special case where the Bregman potential is the Shannon entropy, our algorithm computes the Nash-equilibrium of a regularized bilinear approximation and uses it for a \emph{multiplicative} update. 
The resulting \emph{competitive multiplicative weights (CMW, Algorithm~\ref{alg:cmw})} is a competitive extension of the multiplicative weights update rule that accounts for the interaction between the two players

More generally, our method is based on the 
geometric interpretation of mirror descent proposed by \citep{raskutti2015information}.
From this point of view, mirror descent solves a quadratic local problem in order to obtain a \emph{direction} of movement.
The next iterate is then obtained by moving into this direction for a unit time interval.
Crucially, this notion of \emph{moving into a direction} is not derived from the Euclidean structure of the constraint set, but from the dual geometry defined by the Bregman potential.
In the case of the Shannon entropy, this amounts to moving on straight lines in logarithmic coordinates, resulting in multiplicative updates (see Figure~\ref{fig:dual-geometry-intro}).
This formulation is extended to CGD by letting both agents choose a direction of movement according to a local bilinear approximation of the original problem and then using the dual geometry of the Bregman potential to derive the next iterate.
The resulting \emph{competitive mirror descent (CMD, Algorithm~\ref{alg:cmd})} combines the computational efficiency of CGD with the ability to use general Bregman potentials to account for the nonlinear structure of the constraints.

\begin{algorithm}[H]
    \begin{algorithmic}[1]
        \FOR{$0 \leq k < N$}
           \STATE $\Delta x = -\left(\left[D^2\psi\right] - \left[D_{xy}^2 f\right] \left[D^2 \phi \right]^{-1} \left[D_{yx}^2 g\right] \right)^{-1} \left(  \left[\nabla_x f \right] - \left[D_{xy}^2 f \right] \left[D^2 \phi\right]^{-1} \left[\nabla_y g\right] \right)$
           \STATE $\Delta y = -\left(\left[D^2\phi\right] - \left[D_{yx}^2 g\right] \left[D^2 \psi \right]^{-1} \left[D_{xy}^2 f\right] \right)^{-1} \left( \left[\nabla_x g \right] - \left[D_{yx}^2 g \right] \left[D^2 \psi\right]^{-1} \left[\nabla_x f\right]  \right)$
           \STATE $x_{k + 1} = \left(\nabla \psi\right)^{-1}\left(\left[\nabla \psi(x)\right] + \Delta x\right)$
           \STATE $y_{k + 1} = \left(\nabla \phi\right)^{-1}\left(\left[\nabla \phi(y)\right] + \Delta y\right)$
  			\ENDFOR
  			\RETURN $x_{N}, y_{N}$
	\end{algorithmic}
	\caption{\label{alg:cmd} Competitive mirror descent (CMD) for general Bregman potentials $\psi$ and $\phi$.}
\end{algorithm}


\section{Simplifying constraints by duality}

\footnotetext{Here, $\exp$ is applied element-wise and $\diag_{z}$ denotes the diagonal matrix with entries given by the vector $z$. $\nabla_x f, [D_{xy}^2 f], \nabla_y g, [D_{yx}^2 g]$ denote the gradients and mixed Hessians of $f$ and $g$, evaluated in $(x_{k}, y_{k})$.}
\paragraph{Constrained competitive optimization:} 
The most general class of problems that we are concerned with is of the form of Equation~\eqref{eqn:intro-cco}
where $\C \subset \Reals^m, \K \in \Reals^n$ are convex sets, $\tf: \C \longrightarrow \Reals^{\tn}$ and $\tg: \K \longrightarrow \Reals^{\tm}$ are continuous and piecewise differentiable multivariate functions of the two agents' decision variables $x$ and $y$ and $\tC, \tK$ are closed convex cones.
This framework is extremely general and by choosing suitable functions $\tf, \tg$ and convex cones $\tC, \tK$ it can implement a variety of nonlinear equality, inequality, and positive-definiteness constraints.
While there are many ways in which a problem can be cast into the above form, we are interested in the case where the $f, \tf, g, \tg$ are allowed to be \emph{complicated}, for instance given in terms of neural networks, while the $\tC, \tK$ are simple and well-understood.
For convex constraints and objectives $f$ and $g$ the canonical solution concept is a \emph{Nash equilibrium}, a pair of feasible strategies $\left(\bar{x}, \bar{y}\right)$ such that $\bar{x}$ ($\bar{y}$) is the optimal strategy for $x$ ($y$) given $y = \bar{y}$ ($x = \bar{x}$).
In the non-convex case it is less clear what should constitute a solution and it has been argued \citep{schafer2019implicit} that meaningful solutions need not even be local Nash equilibria.

\paragraph{Lagrange multipliers lead to linear constraints:}
Using the classical technique of Lagrangian duality, the complicated parameterization $f, \tf, g, \tg$ and the simple constraints given by the $\tC, \tK$ can be further decoupled.
The polar of a convex cone $\K$ is defined as $\K^{\circ} \defeq \left\{y : \sup_{x \in \K} x^{\top} y \leq 0 \right\}$.
Using this definition, we can rewrite Problem~\eqref{eqn:intro-cco} as 
\begin{equation}
  \label{eqn:constraint-as-max}
  \min_{\substack{x \in \C, \\ \mu \in \tK^{\circ}}} f(x, y) + \max_{\nu \in \tC^{\circ}} \nu^{\top} \tf(x), 
  \quad \min_{\substack{y \in \K, \\ \nu \in \tC^{\circ}}} g(x, y) + \max_{\mu \in \tK^{\circ}} \mu^{\top} g(y).
\end{equation}
Here we used the fact that the maxima are infinity if any constraint is violated and zero, otherwise.

\paragraph{Watchmen watching watchmen:}
We can now attempt to simplify the problem by making $\mu_j$ ($\nu_i$) decision variables of the $y$ ($x$) player and adding a zero sum objective to the game that incentivizes both players to enforce each other's compliance with the constraints, resulting in 
\begin{equation}
  \label{eqn:constrainedCompLagrange}
  \min_{\substack{x \in \C, \\ \mu \in \tK^{\circ}}} f(x, y) + \nu^{\top} \tf(x)  -  \mu^{\top} \tg(y), 
  \quad \min_{\substack{y \in \K, \\ \nu \in \C^{\circ}}} g(x, y) + \mu^{\top} \tg(y)  - \nu^{\top} \tf(x).
\end{equation}
If Problem~\ref{eqn:intro-cco} is convex and strictly feasible (Slater's condition), its Nash equilibria are equal to those of Problem~\ref{eqn:constrainedCompLagrange} (see supplementary material for details).
\paragraph{A simplified problem:} 
While this is not true in general, we propose to use Problem~\ref{eqn:constrainedCompLagrange} as a more tractable approximation of Problem~\ref{eqn:intro-cco}.
In the following, we assume that the nonlinear constraints of the problem have already been eliminated, leaving us (for possible different $\C$ and $\K$) with 
\begin{equation}
  \label{eqn:problem-reduced}
  \min \limits_{x \in \C} f(x, y), \quad \min \limits_{y \in \K} g(x, y).
\end{equation}

\section{Projected competitive gradient descent suffers from empty threats}
\label{sec:empty-threats}

\paragraph{Competitive gradient descent (CGD):} 
\citep{schafer2019competitive} propose to solve unconstrained competitive optimization problems by choosing iterates $(x_{k + 1}, y_{k + 1})$ as Nash equilibria of a quadratically regularized bilinear approximation
\begin{equation} 
  \label{eqn:CGD-localgame}
  \begin{split}
    x_{k + 1} = x_{k} + &\argmin_{x \in \Reals^m} \left[D_x f\right] x + y^{\top} \left[D_{yx } f\right] x + \left[D_y f\right] y + \frac{x^\top x}{2 \eta} \\
    y_{k + 1} = y_{k} + &\argmin_{y \in \Reals^n} \left[D_x g\right] x + x^{\top} \left[D_{xy} g\right] y + \left[D_y g\right] y + \frac{y^\top y}{2 \eta},
  \end{split}
\end{equation}
where $f,g:\Reals^m \times \Reals^n \longrightarrow \Reals$ are the loss functions of the two agents and $\left[D_x f\right]$, $\left[D_x g\right]$, $\left[D_y f\right]$, $\left[D_y g\right]$, $\left[D_{yx} f\right]$, and $\left[D_{xy} g\right]$ their (mixed) derivatives evaluated in the last iterate $(x_{k}, y_{k})$.

\paragraph{A simple minmax game:} We will use the following simple example to illustrate the difficulties when dealing with inequality constraints.
\begin{equation*}
  \label{eqn:example-unconstrained}
  \min_{x \in \Reals} 2 xy - (1 - y)^2 \quad \min_{y \in \Reals} - 2 xy + (1 - y)^2.
\end{equation*}
When applying CGD to this example, it finds the (unique) Nash-equilibrium $(x, y) = (-1, 0)$.
\paragraph{Adding positivity constraints:} Let us now assume that $x$ and $y$ are constrained to be positive
\begin{equation}
  \label{eqn:example-constrained}
  \min_{x \in \Reals_+} 2 xy - (1 - y)^2 \quad \min_{y \in \Reals+} - 2 xy + (1 - y)^2.
\end{equation}
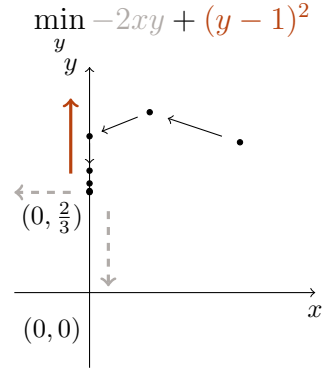
\begin{wrapfigure}{}{4.5cm}
  \vspace{-5pt}
  \begin{tikzpicture}
    \input{figures/tikz/pcgd_failure.tex}
  \end{tikzpicture}
  \caption{\label{fig:pcgd_failure} Since $x \geq 0$, the bilinear term should only lead to $y$ picking a larger value. But CGD is oblivious to the constraint and $y$ decreases in anticipation of $x$ turning negative.}
  \vspace{-10pt}
\end{wrapfigure}
It is easy to see that for any $y \geq 0$ the best strategy of the $x$-player is $x = 0$. 
Given $x = 0$ the best strategy for the $y$-player is $y = 1$, leaving us with the Nash-equilibrium $(x,y) = (0, 1)$.
How do we incorporate the positivity constraints into CGD?
The simplest approach is to combine CGD with projections onto the constraint set, which we call \emph{projected competitive gradient descent (PCGD)}. 
Here, we compute the updates of CGD as per usual but after each update we project back onto the constraint set through $(x,y) \mapsto (\max(0, x), \max(0,y))$.
This generalizes projected gradient descent, a popular and simple method in constrained optimization with proven convergence properties \citep{iusem2003convergence}.

\paragraph{PCGD and empty threats:} Applying PCGD (with $\eta =0.25$) to our example, we see that it converges instead to $(x, y) = (0, 2/3)$, which is suboptimal for the $y$-player!
Since problem~\eqref{eqn:example-constrained} is a convex game, any sensible algorithm should converge to the unique Nash-equilibrium, making this a failure case for PCGD.\\
The explanation of this behavior lies in the conflicting rules of the unconstrained local game~\ref{eqn:CGD-localgame} and the constrained global game~\ref{eqn:example-constrained}.
The local game of Equation~\eqref{eqn:CGD-localgame} in the point $(0,2/3)$ is given by 
\begin{equation*}
  \min_{x} \frac{4x}{3} + 2 x y + \frac{2y}{3} + 2 x^2, \quad \min_{y} -\frac{4x}{3} - 2 x y - \frac{2y}{3} + 2 y^2
\end{equation*}
If $x$ were to stay in zero, the optimal strategy for $y$ under the local game would be $y = 1/6$, moving it towards $1$, the global optimum.
However, the local game does not have constraints, incentivizing $x$ to turn negative.
This is an \emph{empty threat} in that it violates the constraints and will therefore be undone by the projection step.
However, this information is not included in the local game and the outlook of $x$ becoming negative deters $y$ from moving towards the optimal strategy of $y=1$ (c.f. Figure~\ref{fig:pcgd_failure}).
We note that this problem is generic for projected versions of algorithms featuring \emph{competitive terms} involving the mixed Hessian.
Since \citep{schafer2019competitive} identify these terms as crucial for convergence, this is a major obstacle for constrained competitive optimization.

\section{Mirror descent and Bregman potentials}
\paragraph{Bregman potentials:} The underlying reason for the \emph{empty threats-phenomenon} described in the last section is that a local method such as CGD has no way on knowing how close it is to the boundary of the constraint set.
In single-agent optimization this problem has been addressed by the use of \emph{Bregman potentials} or \emph{Barrier functions}.

\begin{definition}
  We call a strictly convex and two-times differentiable function $\psi:\C \longrightarrow \Reals $ a \emph{Bregman potential} on the convex domain $\C$.
\end{definition}

\begin{table}
  \centering
  \begin{tabular}{c|c|c|c|c}
    \rowcolor{lightsilver}
    $\C$ & Potential $\psi(p)$ & $\left[D\psi(p)\right] x $ & $x^{\top} \left[D^2\psi(p)\right] x $ & Divergence $\Div{\psi}{p}{q}$\\
    \hline 
    \hline 
    $\Reals^m$ & $\frac{p^{\top} A \ p}{2}$, $A \in \Herms_+^{m \times m}$ & $p^{\top}A x $ & $x^{\top} A x$ & $\frac{\left(p - q\right)^{\top} A \ \left(p - q\right)}{2}$ \\
    \hline
    \rowcolor{lightsilver}
    $\Reals_+^m$ & $\sum \limits_{i} p_i \log(p_i) - p_i$ & $ \sum \limits_{i} \log(p_i)  x_i$ & $\sum \limits_{i} \frac{x_{i}^2}{p_{i}}$  & $\sum \limits_{i} \left(p_i \log(\frac{p_i}{q_i}) - p_i + q_i \right)$   \\
    \hline 
    $\Reals_+^m$ & $\sum \limits_i -\log(p_i)$ & $ - \sum \limits_{i} \frac{x_{i}}{p_{i}}$ & $\sum_{i} \frac{x_{i}^2}{p_{i}^2}$ & $\sum \limits_i \left( \frac{p_i}{q_i} - \log\left(\frac{p_i}{q_i}\right) - 1 \right)$   \\
    \hline
    \rowcolor{lightsilver}
    $\Herms_+^m$ & $ -\logdet(p)$ & $\tr\left[p^{-1} x \right]$  & $\tr\left[xp^{-1} p^{-1}x\right] $ & $ \tr\left[p / q - \Id \right]  - \logdet(p/q)$ \\
    \hline
  \end{tabular}
  \caption{\label{tab:bregman} Some Bregman potentials, their domains, derivatives, and Bregman divergences.}
  \vspace{-15pt}
\end{table}
In Figure~\ref{tab:bregman} we have shown some popular Bregman potentials and their respective domains.
For the purposes of this work, a Bregman potential $\psi$ on $\C$ is best understood as quantifying how close different points $y \in \C$ are from $\argmin \psi$, which we think of as the center of $\C$.
Importantly, the notion of distance induced by $\psi$ is anisotropic, and usually increases rapidly as $y$ approaches the boundary of the domain.
In particular, derivative information of $\psi$ at a point $p$ allows us to infer its position relative to the boundary of $\C$, allowing a local algorithm to take into account the global structure of the constraint set.

\paragraph{Bregman divergences:} 
Associated with a Bregman potential $\psi(\cdot)$ on a domain $\C$ is the Bregman divergence $\Div{\psi}{\cdot}{\cdot}$ that allows us to extend the notion of distance between $y$ and $\argmin \psi$ given by $\psi$ to a notion of distance between arbitrary elements of $\C$.

\begin{definition}
  The Bregman divergence associated to the potential $\psi$ is defined as 
  \begin{equation}
    \Div{\psi}{p}{q} \defeq \psi(p) - \psi(q) - [\nabla \psi(q)] (p - q).
  \end{equation}
\end{definition}

Just like a squared distance, $\Div{\psi}{p}{q}$ is convex, positive, and achieves its minimum of zero if and only if $p = q$. 
However, it is not symmetric in the sense that $\Div{\psi}{p}{q} \neq \Div{\psi}{q}{p}$, in general.
Crucially, a Bregman divergence will in general not be translation invariant and can therefore take the position relative to the boundary of $\C$ into account.

\paragraph{Mirror Descent:}
\emph{Mirror descent} \citep{nemirovsky1983problem} uses a Bregman to improve gradient descent by the following update rule.
\begin{equation}
    \label{eqn:update-md}
    x_{k + 1} = \argmin_{x} \left[ D f\left(x_{k}\right)\right]\left(x - x_{k}\right) + \Div{\psi}{x}{x_{k}}. 
\end{equation}
By solving for the first order optimality conditions, the update can be computed in closed form as  
\begin{equation}
    x_{k + 1} = \left(D\psi\right)^{-1} \left( \left[D\psi\left(x_{k}\right)\right]   - \left[Df\left(x_{k}\right)\right]  \right),
\end{equation}
where $\left(D \psi\right)^{-1}(y)$ is defined as the point $x \in \C$ such that $\left[D \psi(x)\right] = y$.
By strict convexity of $\psi$, if $\left(D \psi\right)^{-1}(y)$ exists it is unique. 
In this case, $\left(D \psi\right)^{-1}(y)$ is contained in $\C$ which ensures that mirror descent satisfies the constraints without an additional projection step.
\begin{definition}
    A Bregman potential $\psi$ is \emph{complete} on $\C$ if the map $D\psi :\C \longrightarrow \Reals^{1 \times m}$ is surjective.
\end{definition}
If $\psi$ is complete, the mirror descent update is well-defined for all gradients. 
From this point of view, the necessity for a projection step in inequality constrained gradient descent arises from the potential $\psi(x) = x^{\top}x / 2$ not being complete on the constraint set $\C$.
A popular choice of potential that is complete on the positive orthant $\Reals_+^m$ is the Shannon entropy $\psi(x) \sum_i x_i \log\left(x_i\right) - x_i$ resulting in \emph{entropic mirror descent}
\begin{equation}
    x_{k + 1} = \exp\left( \log\left(x_{k}\right) - \left[Df\left(x_{k}\right)\right]^{\top}  \right) = x_{k} \exp\left(\left[Df\left(x_{k}\right)\right]^{\top}\right).
\end{equation}
Here, multiplication, exponentiation, and logarithms are defined component-wise.

\paragraph{Naive competitive mirror descent:}
A possible generalization of mirror descent to Problem~\eqref{eqn:problem-reduced} is
\begin{equation*} 
  \label{eqn:NCGD-localgame}
  \begin{split}
    x_{k + 1} = &\argmin_{x \in \Reals^m} \left[D_x f\right] \left(x - x_{k}\right) + \left(y - y_{k}\right)^{\top} \left[D_{yx } f\right] \left(x - x_{k}\right) + \left[D_y f\right] \left(y - y_{k}\right) + \Div{\psi}{x}{x_{k}} \\
    y_{k + 1} = &\argmin_{y \in \Reals^n} \left[D_x g\right] \left(x - x_{k}\right) + \left(x - x_{k}\right)^{\top} \left[D_{xy} g\right] \left(y - y_{k}\right) + \left[D_y g\right] \left(y - y_{k}\right) + \Div{\phi}{y}{y_{k}},
  \end{split}
\end{equation*}
where $\psi$ and $\phi$ are complete Bregman potentials on $\C$ and $\K$, respectively. 
However, the local game in this update rule does not have a closed form solution as in mirror descent or competitive gradient descent.
Instead, it requires us to solve a nonlinear competitive optimization problem at each step.
While it is possible to alternatingly solve for the two players' strategies until convergence, this will be significantly slower than the Krylov subspace methods employed to solve system of linear equations in the original CGD \citep{schafer2019competitive}.
In order to avoid this overhead, we will now develop an alternative competitive generalization of mirror descent rooted in its information-geometric interpretation.

\section{The information geometry of Bregman divergences}

\paragraph{Reminder on differential geometry:}
To present the material in this chapter, we briefly review the following basic notions of differential geometry. 
\begin{definition}[Submanifold of $\Reals^m$]
  {$\M \subset \Reals^m$ is a $k$-dimensional smooth submanifold of $\Reals^m$ if for every point $p \in \mathcal{M}$ there exists an open ball $B(p)$ centered at $p$ and a smooth function $G_p:B(p) \longrightarrow \Reals^{m-k}$, such that the rank of the Jacobian of $G_p$ is $k$ everywhere and $\mathcal{M} \cap B_p = G^{-1}(0)$.
  \label{def:submanifold}}
\end{definition}
We can think of a smooth submanifold as a (possible curved) surface in $\Reals^m$.
\begin{definition}[Tangent space]
  {The tangent space $\mathcal{T}_{p}\mathcal{M}$ of a $k$-dimensional submanifold $\mathcal{M}$ of $\Reals^m$ in $p$ is given by the null space of the Jacobian of $G_p$ in $p$. Here, $G_p$ is chosen as in Definition~\ref{def:submanifold}.}
\end{definition}
To a creature living on $\M$, the elements of the tangent space in $p \in \M$ correspond to the velocities with which it could depart from $p$.
\begin{definition}[Riemannian metric]
  A Riemannian metric is a map $p \mapsto g_p(\cdot, \cdot)$ that assigns to each $p \in \mathcal{M}$ an inner product on $\mathcal{T}_p \mathcal{M}$.
\end{definition}
If the elements $x \in \T_p \M$ of the tangent space are velocities, $g_{p}\left(x, x\right)$ is their (squared) speed. 
It allows us to compare the magnitude or significance of moving in different velocities.
\begin{definition}[Exponential map]
  An \emph{exponential map} is a collection $\left\{\Exp_p: \mathcal{T}_p \mathcal{M} \longrightarrow \mathcal{M} \right\}_{p \in \mathcal{M}}$ that satisfies 
  \begin{equation*}
    \label{eqn:exponential-property}
    \Exp_p\left(\left( t + s \right) x\right) = \Exp_q\left(s y\right),  \quad \mathrm{for}\ q = \Exp_p(tx), \quad y = \frac{\diff}{\dr} \left.\Exp_{p}\left(r x \right))\right|_{r = t}.
  \end{equation*}
\end{definition}
\begin{wrapfigure}{}{0.40\textwidth}
  \vspace{-10pt}
  \centering
  \includegraphics[width=0.40\textwidth]{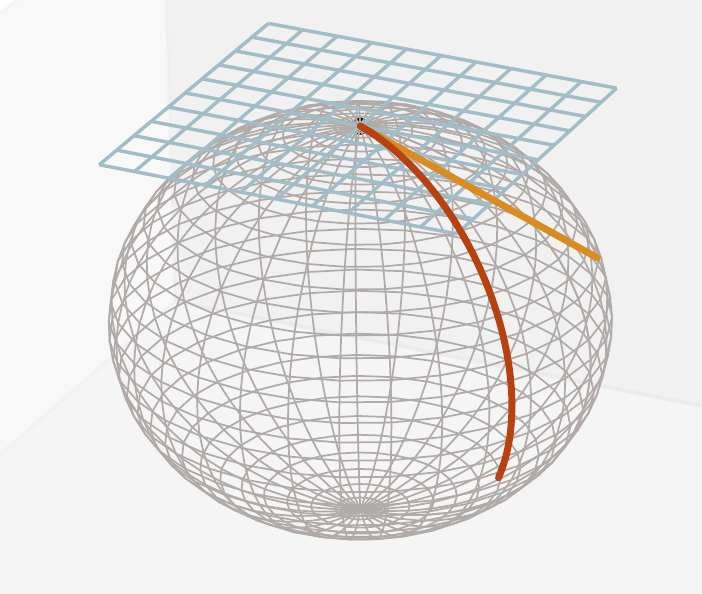}
  \caption{\label{fig:manifold}  {\color{darksilver} A manifold $\M$}, {\color{darklightblue} tangent space $\T_p \M$}, {\color{darkorange} tangent vector $x \in \T_p \M$}, and the path described by the {\color{darkrust} exponential map $\left\{q: q = \Exp_p(t x), \mathrm{for}\ t \in [0,1]\right\}$}.}
  \vspace{-20pt}
\end{wrapfigure}
The exponential map $\Exp_p(x)$ returns the destination reached when \emph{walking straight} in the velocity $x$ for a unit time interval, starting in $p$.

\paragraph{The geometry of Bregman potentials:}
We will now explain how a Bregman potential equips its domain $\C$ with a geometric structure (see \citep{amari2007methods,amari2016information} for details).
Our manifold $\M$ will simply be the interior of $\C \subset \Reals^{m}$ with $\T_p \M$ given by $\Reals^m$ for all $p \in \M$ (without loss of generality, we assume that 
$\interior \C$ has full dimension). 
The Riemannian metric $g^{\psi}$ associated with the potential $\psi$ is given by 
\begin{equation*}
  g_p^{\psi}\left(x, x\right) = \frac{x^{\top} \left[D^2_{xx} \psi \left( p \right)\right] x}{2} = \frac{\diff^2}{2\dr^2} \Div{\psi}{p + rx}{p}.
\end{equation*}
Following the interpretation of the Bregman divergence as a notion of squared distance, the metric $g_{p}^\psi$ measures how quickly a given velocity $x \in \T_p\M$ will take us away from $p$, according to this notion of distance.

\paragraph{The dual exponential map:}
A key feature of Bregman potentials is that they induce a new exponential map on the manifold $\C$. 
$\M$ is an open subset of $\Reals^m$ and thus a restriction of the Euclidean exponential map $\Exp_{p}\left(x\right) \defeq p + x$ to $p \in \M$ is an exponential map on $\mathcal{M}$, which we call the \emph{primal} exponential map.
However, a Bregman potential $\psi$ also induces a \emph{dual} exponential map on $\M$, given by 
\begin{equation*}
  \Exp_p^{\psi}\left(x\right) = \left(D\psi\right)\left(\left(D\psi\right)^{-1} + \left[D^2 \psi\left(p\right)\right] x \right)
\end{equation*}
In the special case of $\C = \Reals_+^m$ and $\psi(x) = \sum_i x_i \log(x_i) - x_i$ given by the Shannon entropy, the dual exponential map is given by
\begin{equation*}
  \left(\Exp_p^{\psi}\left(x\right)\right)_{i} = \exp\left( \log(p_i) + \frac{x_i}{p_i}\right) = p_i \exp\left(\frac{x_i}{p_i}\right).
\end{equation*}
Thus, while the \emph{primal straight lines} $t \mapsto \Exp_{x}\left( t x\right)$ have constant additive rate of change, \emph{dual straight lines} have constant $t \mapsto \Exp_{x}\left( t x\right)$ have constant relative rate of change.

An important property of the dual exponential map it that it inherits the completeness property of $\psi$.
\begin{lemma}
  If the potential $\psi$ is complete with respect to $\M$, the dual exponential map $\Exp^{\psi}$ is complete in the sense that 
  \begin{equation}
    \forall p \in \M, \forall x \in \Reals^m: \Exp_{p}^{\psi}\left(x\right) \in \M.
  \end{equation}
\end{lemma}
Thus, a complete potential $\psi$ provides us with way of following a direction $x \in \T_{p}\M$ in $\M = \interior \C$ that ensures that we never accidentally leave the feasible set.
We will now recast mirror descent in terms of the dual geometry induced by $\psi$.

\section{Competitive mirror descent}
In \citep{raskutti2015information}, it is observed that mirror descent has a natural formulation in terms of the dual geometry induced by $\psi$. 
Namely, its updates can be expressed as

\begin{enumerate}[wide, labelwidth=5pt, labelindent=5pt]
  \item Choose a direction of movement that minimizes a linear local approximation, regularized with the metric induced by $\psi$.
    \begin{equation*}
      \Delta x = \argmin \limits_{x \in \mathcal{T}_{x_{k}} \mathcal{M}} [D_xf] x + \frac{1}{2} x^T \left[D^2\psi(x_k)\right] x
    \end{equation*}

  \item
    Compute $x_{k+1}$ by moving into this direction, according to the \emph{dual} geometry of $\psi$. 
    \begin{equation*}
      x_{k+1} = \Exp_{x_{k}}^{\psi}\left(\Delta x\right)
    \end{equation*}
\end{enumerate}

For $\psi$ ($\phi$) a complete Bregman divergence on $\C$ ($\K$), this form of mirror descent can be readily extended to Problem~\eqref{eqn:problem-reduced}, resulting in \emph{competitive mirror descent (CMD)}: 
\begin{enumerate}[wide, labelwidth=5pt, labelindent=5pt]
  \item Solve for a local Nash equilibrium where both players try to minimize the bilinear local approximation of their objective, regularized with the metrics induced by $\psi$ and $\phi$.
  \begin{align*}
    \label{eqn:cmd-localgame}
    \Delta x = \argmin \limits_{x \in \mathcal{T}_{x_{k}} \C} &  \left[D_x f\right]x + x^{\top} \left[D_{xy}^2 f\right] y + \left[D_{y} f\right] y + \frac{1}{2} x^{\top} \left[D^2 \psi \right] x \\
    \Delta y = \argmin \limits_{y \in \mathcal{T}_{y_{k}} \K} &  \left[D_x g\right]x + x^{\top} \left[D_{xy}^2 g\right] y + \left[D_{y} g\right] y + \frac{1}{2} y^{\top} \left[D^2 \phi \right] y 
  \end{align*}

  \item Compute $x_{k+1}$ ($y_{k+1}$) by moving into this direction according to the dual geometries of $\psi$ ($\phi$).
  \begin{equation*}
    x_{k+1} = \Exp^{\psi_{\C}}_{x_k}\left(\Delta x\right), \quad y_{k+1} = \Exp^{\phi}_{y_k}\left(\Delta y\right)
  \end{equation*}
\end{enumerate}

Since $\psi$ and $\phi$ are complete, each iterate is guaranteed to be feasible, while the local game~\eqref{eqn:CGD-localgame} is quadratic and can be solved in closed form, resulting in algorithm~\ref{alg:cmd}.

\section{Numerical implementation and experiments}
\begin{wrapfigure}{}{0.36\textwidth}
  \vspace{-15pt}
  \centering
  \includegraphics[width=0.17\textwidth]{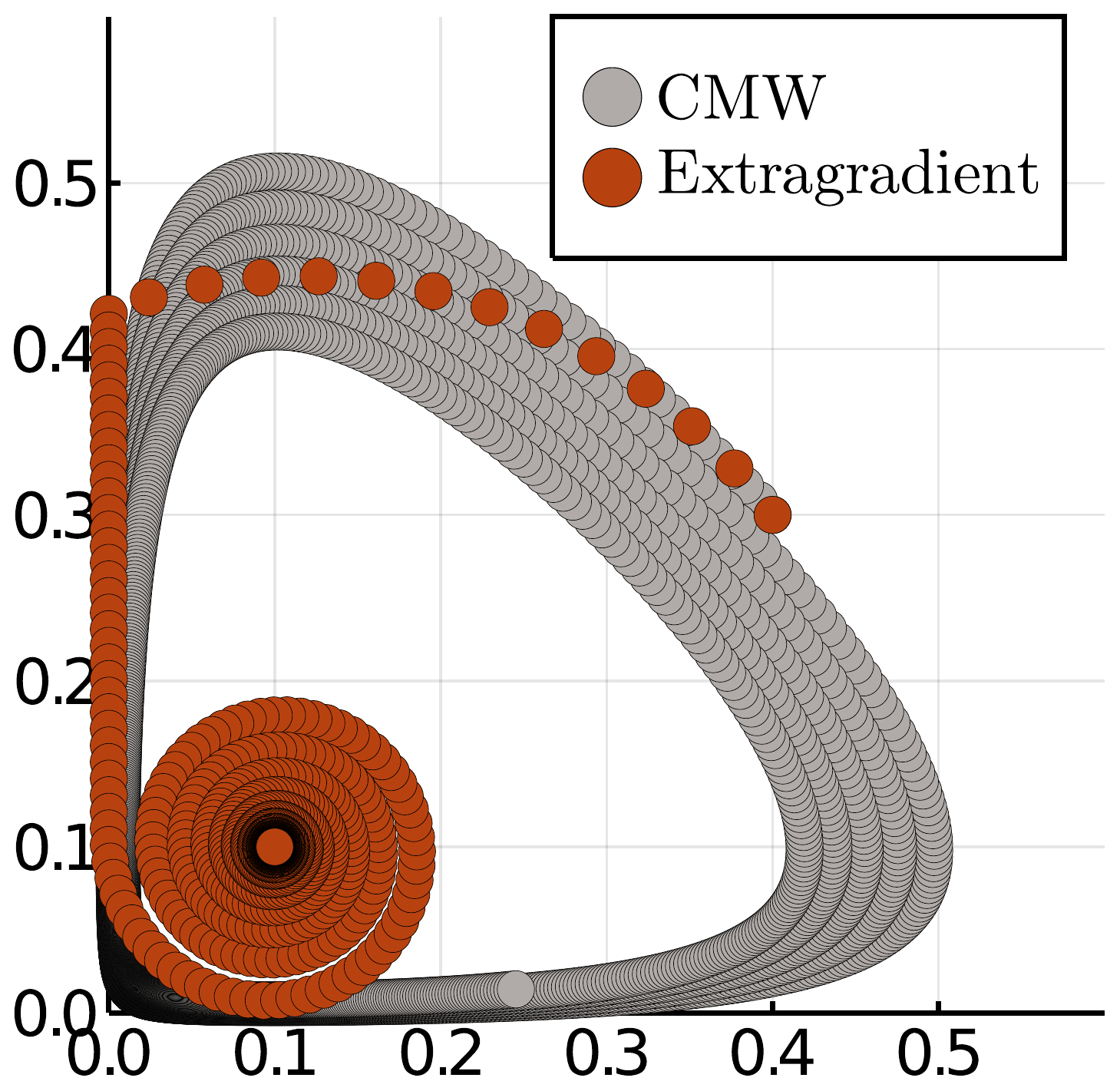}
  \includegraphics[width=0.17\textwidth]{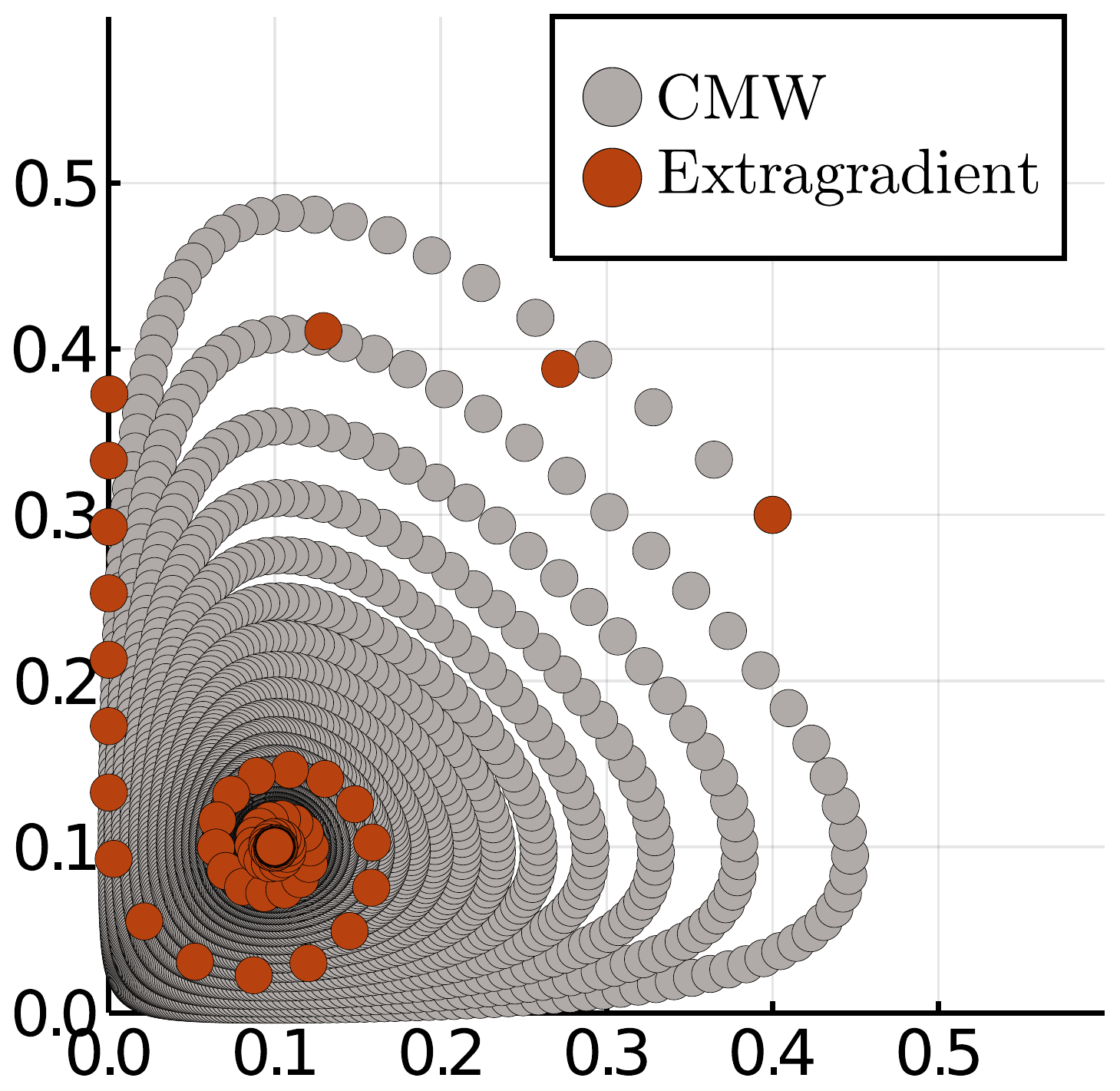}
  \includegraphics[width=0.17\textwidth]{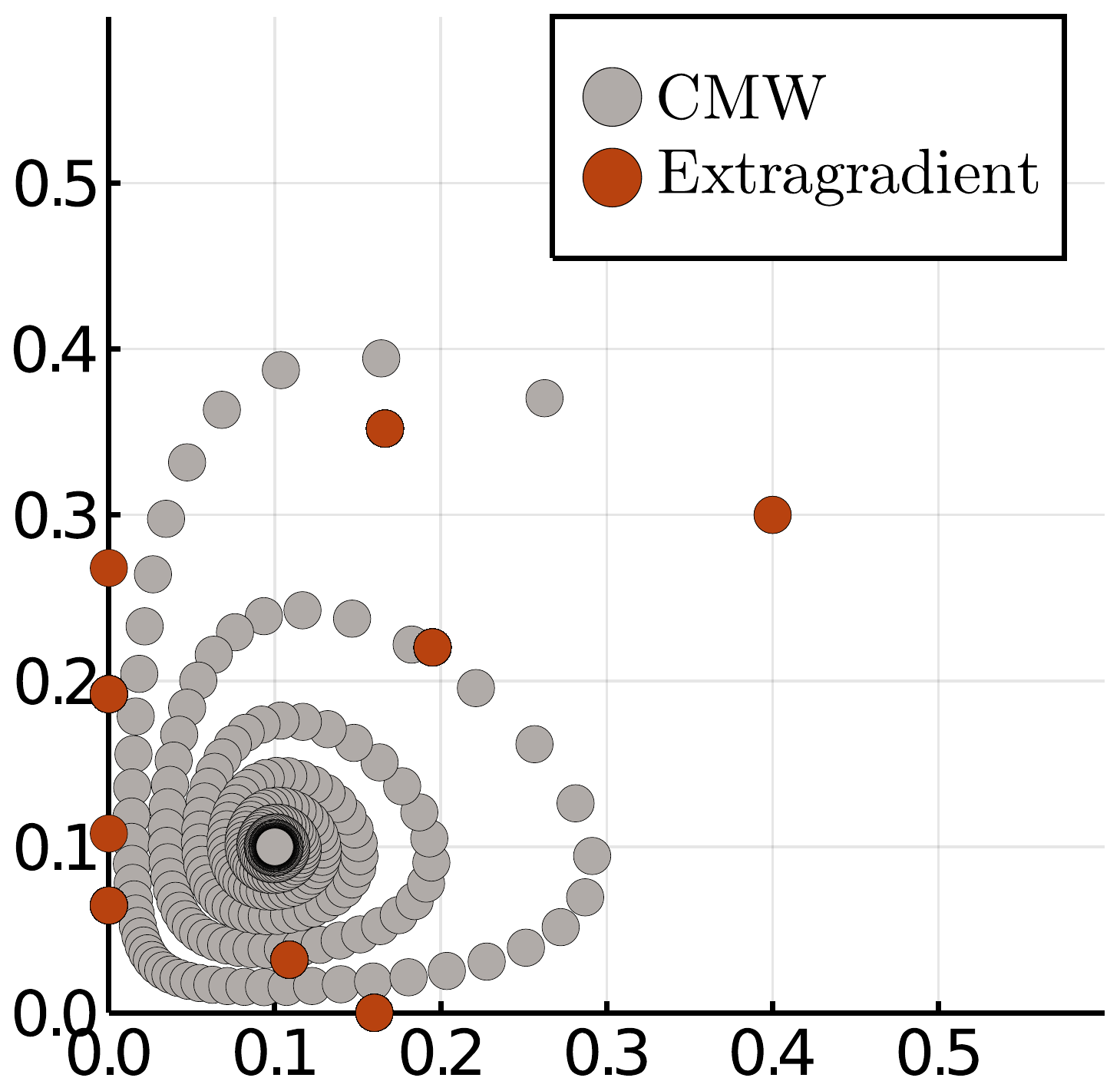}
  \includegraphics[width=0.17\textwidth]{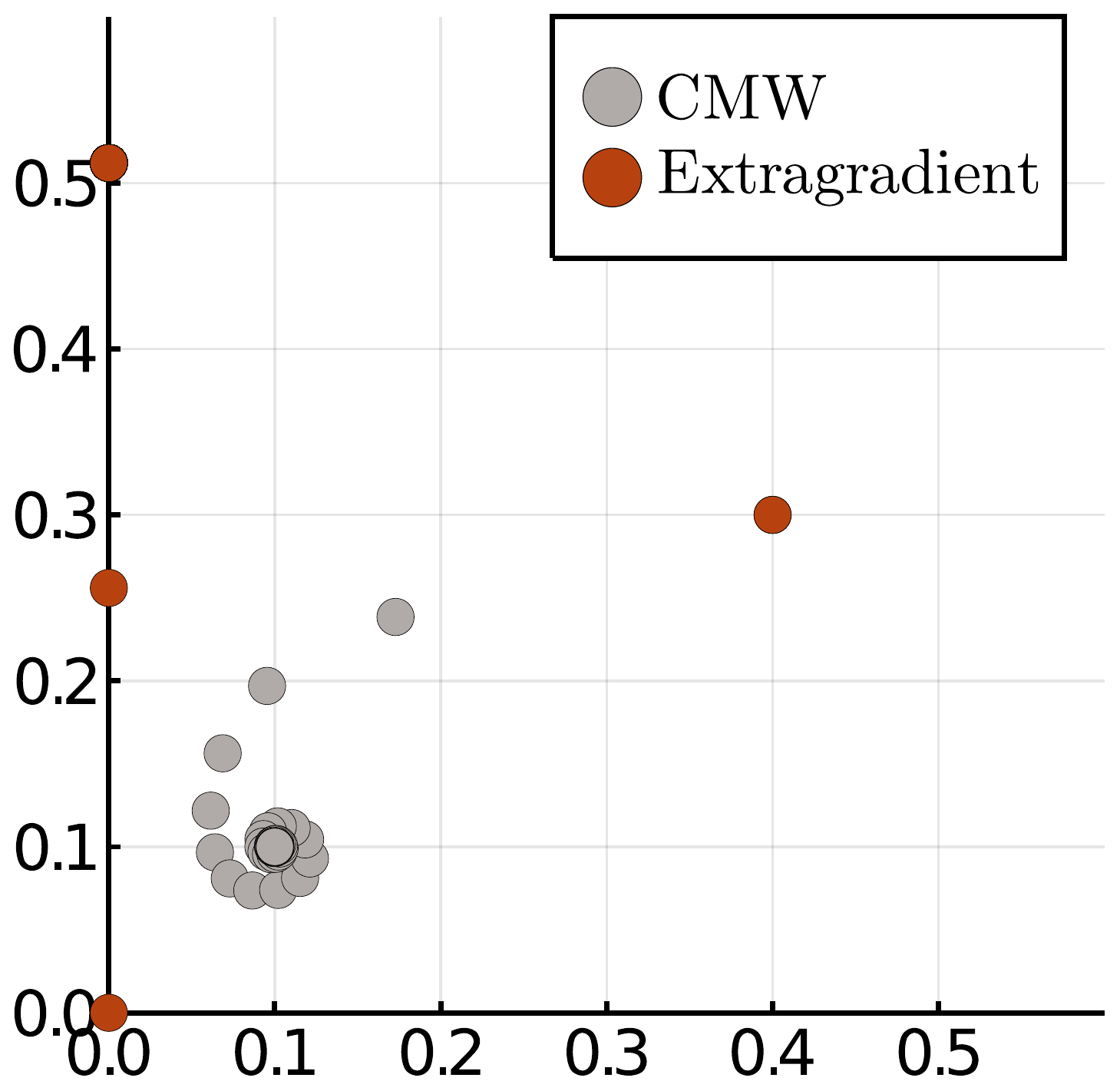}
  \caption{\label{fig:comparison} CMW and PX applied to $f(x,y) = \alpha (x-0.1)(y-0.1) = -g(x,y)$ ($\alpha \in \{0.1, 0.3, 0.9, 2.7\}$). For small $\alpha$, PX converges faster but for large $\alpha$ it diverges.}
  \vspace{-25pt}
\end{wrapfigure}
We defer the discussion of numerical implementation and numerical experiments to the supplementary material, noting that as discussed in \citep{schafer2019competitive,schafer2019implicit}, the matrix inverse in CMD does not prevent it from scaling to large problems.

As discussed in \citep{schafer2019competitive}[Section 3], the \emph{competitive term} involving the mixed Hessian $D_{xy}^2f, D_{yx}^2g$ is crucial to ensure convergence for bilinear problems. 
Most methods that include a competitive term such as \citep{mescheder2017numerics,balduzzi2018mechanics,letcher2019differentiable,gemp2018global} encounter the empty threats-phenomenon described in Section~\ref{sec:empty-threats} when combined with a projection.
A notable exception is the projected extragradient method (PX) \citep{korpelevich1977extragradient}(see also \citep{facchinei2003finite}[Chapter 12]) which we therefore see as the main competitor to CMD.
The main practical advantage of CMD over the projected extragradient method is that just as CGD, it can handle strong interactions between the two players (large $D_{xy}^2f, D_{yx}^2g$), without needing to reduce its step size. 

\section{Conclusion}
In this work we combine competitive gradient descent \citep{schafer2019competitive} and mirror descent \citep{nemirovsky1983problem} to obtain a general algorithm for constrained competitive optimization.
Our method uses ideas from information geometry \citep{raskutti2015information,amari2007methods,amari2016information} to derive an update rule that incorporates the interaction between the players and the nonlinear structure of the constraint set, while only solving a linear system of equations at each iteration.
There are numerous directions for future work including the development of a convergence theory, using a multilinear local game to extend CMD to more than two players, and studying its application to more general conic constraints.

\section*{Broader impact}
This work proposes a basic optimization method and thus its societal impact is hard to predict. We hope that by providing better tools for constrained optimization we can provide practitioners with more flexibility in designing methods for reliable machine learning.

\begin{ack} 
AA is supported in part by Bren endowed chair, DARPA PAIHR00111890035, LwLL grants, Raytheon, BMW, Microsoft, Google, Adobe faculty fellowships, and DE Logi grant.
FS gratefully acknowledges support by the Ronald and Maxine Linde Institute of Economic and Management Sciences at Caltech.
FS and HO gratefully acknowledge support by  the Air Force Office of Scientific Research under award number FA9550-18-1-0271 (Games for Computation and Learning) and the Office of Naval Research under award number N00014-18-1-2363.
\end{ack}

\medskip
{\small \bibliography{cmd} \par}

\appendix

\section{Simplifying constraints by duality}
\paragraph{Constrained competitive optimization:} 
The most general class of problems that we are concerned with is of the form 
\begin{equation}
  \label{ap:eqn:constrainedComp}
  \min_{\substack{x \in \C, \\ \tf(x) \in \tC}} f(x, y), \quad \min_{\substack{y \in \K, \\ \tg(y) \in \tK}} g(x, y),
\end{equation}
where $\C \subset \Reals^m, \K \in \Reals^n$ are convex sets, $\tf: \C \longrightarrow \Reals^{\tn}$ and $\tg: \K \longrightarrow \Reals^{\tm}$ are continuous and piecewise differentiable multivariate functions of the two agents' decision variables $x$ and $y$ and $\tC, \tK$ are closed convex cones.
This framework is extremely general and by choosing suitable functions $\tf, \tg$ and convex cones $\tC, \tK$ it can implement a variety of nonlinear equality, inequality, and positive-definiteness constraints.
While there are many ways in which a problem can be cast into the above form, we are interested in the case where the $f, \tf, g, \tg$ are allowed to be \emph{complicted}, for instance given in terms of neural networks, while the $\tC, \tK$ are simple and well-understood.
For convex constraints and objectives $f$ and $g$ the canonical solution concept is a \emph{Nash equilibrium}.
\begin{definition}
  \label{ap:def:nash}
  A \emph{Nash equilibrium} of Problem~\ref{ap:eqn:constrainedComp} is a pair of pair of feasible strategies $\left(\bar{x}, \bar{y}\right)$ such that $\bar{x}$ ($\bar{y}$) is the optimal strategy for $x$ ($y$) given $y = \bar{y}$ ($x = \bar{x}$).
\end{definition}
In the non-convex case it is less clear what should constitute a solution and it has been argued \citep{schafer2019implicit} that meaningful solutions need not even be local Nash equilibria.

\paragraph{Lagrange multipliers for linear constraints:}
Using the classical technique of Lagrangian duality, the complicated parameterization $f, \tf, g, \tg$ and the simple constraints given by the $\tC, \tK$ can be further decoupled.
The polar of a convex cone $\K$ is defined as $\K^{\circ} \defeq \left\{y : \sup_{x \in \K} x^{\top} y \leq 0 \right\}$.
Using this definition, we can rewrite Problem~\eqref{ap:eqn:constrainedComp} as 
\begin{equation}
  \label{ap:eqn:constraint-as-max}
  \min_{\substack{x \in \C, \\ \mu \in \tK^{\circ}}} f(x, y) + \max_{\nu \in \tC^{\circ}} \nu^{\top} \tf(x), 
  \quad \min_{\substack{y \in \K, \\ \nu \in \tC^{\circ}}} g(x, y) + \max_{\mu \in \tK^{\circ}} \mu^{\top} g(y).
\end{equation}
Here we used the fact that the maxima are infinity if any constraint is violated and zero, otherwise.

\paragraph{Watchmen watching watchmen:}
We can now attempt to simplify the problem by making $\mu_j$ ($\nu_i$) decision variables of the $y$ ($x$) player and adding a zero sum objective to the game that incentivizes both players to enforce each other's compliance with the constraints, resulting in 
\begin{equation}
  \label{ap:eqn:constrainedCompLagrange}
  \min_{\substack{x \in \C, \\ \mu \in \tK^{\circ}}} f(x, y) + \nu^{\top} \tf(x)  -  \mu^{\top} \tg(y), 
  \quad \min_{\substack{y \in \K, \\ \nu \in \C^{\circ}}} g(x, y) + \mu^{\top} \tg(y)  - \nu^{\top} \tf(x).
\end{equation}

Is there a relationship between the Nash equilibria of Problems~\eqref{ap:eqn:constrainedComp}~and~\eqref{ap:eqn:constrainedCompLagrange}? 
It turns out that this question can be studied in terms of two \emph{decoupled} zero-sum games.
\begin{lemma}
  \label{ap:lem:decoupling}
  A pair of points $\bar{x}$, $\bar{y}$ is a Nash equilibrium of Problem~\ref{ap:eqn:constrainedComp} if and only if $\bar{x}$ and $\bar{y}$ are minimizers of 
  \begin{equation}
    \label{ap:eqn:decoupled-constraint}
    \min_{\substack{x \in \C, \\ \tf(x) \in \tC}} F(x), \quad \min_{\substack{y \in \K, \\ \tg(y) \in \tK}} G(y),
  \end{equation}
  for $F(x) \defeq f(x,\bar{y})$ and $G(y) \defeq g(\bar{x},y)$.
  Similarly, a pair of strategies $(\bar{x}, \bar{\mu})$, $(\bar{y}, \bar{\nu})$ is a Nash equilibrium of Problem~\ref{ap:eqn:constrainedCompLagrange} if and only if $(\bar{x}, \bar{\nu})$, $(\bar{y}, \bar{\mu})$ are Nash equilibria of the decoupled zero sum games 
  \begin{equation}
    \label{ap:eqn:decoupled-lagrange}
    \min_{x \in \C} \max_{\nu \in \tC^{\circ}} F(x) + \nu^{\top} \tf(x), 
    \quad \min_{y \in \K} \max_{\mu \in \tK^{\circ}} G(y) + \mu^{\top} \tg(y).
  \end{equation}
\end{lemma}
\begin{proof}
  The first part of the Lemma follows directly from the Definition~\ref{ap:def:nash} of a Nash equilibrium.
  For the second part we observe that $(\bar{x}, \bar{\mu})$ ($(\bar{y}, \bar{\nu})$) is an optimal strategy against $(\bar{y}, \bar{\nu})$ ($(\bar{x}, \bar{\mu})$) if and only if $\bar{x}$ ($\bar{y}$) minimizes $F$ ($G$) over $\C$ ($\K$) and $\bar{\mu}$ ($\bar{\nu}$) maximizes $\mu \mapsto \mu^{\top} g(\bar{y})$ ($\nu \mapsto \nu^{\top} f(\bar{x})$).
  But this is exactly the definition of $(\bar{x},\bar{\nu})$ ($(\bar{y},\bar{\mu})$) being a Nash equilibrium of Problem~\ref{ap:eqn:decoupled-lagrange}.
\end{proof}

While this result follows directly from the definition, it reduces the question to that of constrained single-agent optimization, which has been studied extensively, allowing us to decuce the following theorem. 
For convex, strictly feasible problems (\emph{``Slater's condition''}), we can show the equivalence of Problems~\eqref{ap:eqn:constrainedComp}~and~\eqref{ap:eqn:constrainedCompLagrange}.
In order to formulate these results in full generality, we need the following definition.
\begin{definition}
We call a function $f: \K \longrightarrow \Reals^n $ \emph{convex with respect to the cone} $\C \subset \Reals^n$ if 
  \begin{equation*}
    \tau f(x) + \left(1 - \tau\right) f(y) - f\left(\tau x + \left(1 - \tau\right) y\right) \in \C, \ \forall \tau \in [0,1],\  x,y \in \K.
  \end{equation*}
\end{definition}
With this definition, we can formulate the following theorem.
\begin{theorem}
    Assume that the following holds:
  \begin{enumerate}[wide, labelwidth=20pt, labelindent=5pt, label=(\roman*):]
    \item $f$, $\tf$, $g$, and $\tg$ are continuous.
    \item $\tf$ ($\tg$) is convex with respect to $-\tC$ ($-\tK$).
    \item For all $\bar{y} \in \K$ ($\bar{x} \in \C$), $F$ ($G$) as defined in Lemma~\ref{ap:lem:decoupling} is convex.
    \item For all $\bar{x} \in \C$ and $\bar{y} \in \K$, the minimal values of Problem~\ref{ap:eqn:decoupled-constraint} are finite (not $- \infty$).
    \item There exist $(x,y)$ such that $x \in \interior \C$, $\tf(x) \in \interior \tC$, $y \in \interior \K$, $\tg(y) \in \interior \tK$.
  \end{enumerate}
  Then, $\bar{x}$ and $\bar{y}$ are a Nash equilibrium of Problem~\eqref{ap:eqn:constrainedComp} if and only if there exist $\bar{\nu}$ and $\bar{\mu}$ such that $(\bar{x}, \bar{\mu})$ and $(\bar{y}, \bar{\nu})$ are a Nash equilibrium of Problem~\eqref{ap:eqn:constrainedCompLagrange}.

\end{theorem}
\begin{proof}
  By Lemma~\ref{ap:lem:decoupling} it is enough to show that $\bar{x}$ and $\bar{y}$ are minimizers of Problem~\ref{ap:eqn:decoupled-constraint} if and only if they can be complemented with Lagrange multipliers $\bar{\nu}$ and $\bar{\mu}$ to obtain Nash equilibria of Problem~\ref{ap:eqn:decoupled-lagrange}.
  This result is shown, for instance, in \citep{ekeland1999convex}[Chapter 3, Theorem 5.1].
\end{proof}
\paragraph{A simplified problem:} 
In the general non-convex setting, the relationship between Problems~\eqref{ap:eqn:constrainedComp}~and~\eqref{ap:eqn:constrainedCompLagrange} is difficult to characterize. 
In this case, Problem~\eqref{ap:eqn:constrainedCompLagrange} serves as an approximation to Problem~\ref{ap:eqn:constrainedComp} that might be easier to solve.
Techniques for closing the duality gap in single agent optimization, such as the addition of redundant constraints, can also serve to improve the approximation of Problem~\eqref{ap:eqn:constrainedComp} by Problem~\ref{ap:eqn:constrainedCompLagrange}.

\section{Numerical implementation and experiments}

\paragraph{Dual coordinate system for improved stability:} 
In principle, either the primal, or the dual coordinate system can be used to keep track of the running iterate. 
However, we observe that storing the iterates in CGD in the dual coordinate system improves the numerical stability of the algorithm.

When expressing the update direction in the dual coordinate system, the local problem reads

\begin{align*}
  \min \limits_{x^* \in \Reals^m} &  \left[D_x f\right] \left[D^2 \psi \right]^{-1} x^* + x^{*,\top} \left[D^2 \psi \right]^{-1} \left[D_{xy}^2 f\right] \left[D^2 \phi \right]^{-1} y^* + \frac{1}{2} x^{*, \top} \left[D^2 \psi \right]^{-1} x^* \\
  \min \limits_{y^* \in \Reals^n} &  \left[D_y g\right] \left[D^2 \phi \right]^{-1} y^* + x^{*,\top} \left[D^2 \psi \right]^{-1} \left[D_{xy}^2 g\right] \left[D^2 \phi \right]^{-1} y^* + \frac{1}{2} y^{*, \top} \left[D^2 \phi \right]^{-1} y^* \\,
\end{align*}

where all derivatives are computed in the last iterate $(x_{k}, y_{k})$
Setting the derivatives with respect to $x*$ ($y^*$) to zero, we obtain  

\begin{align}
  \label{ap:eqn:deriv-zero1}
  & \left[D_x f\right] \left[D^2 \psi \right]^{-1} + \left(\left[D^2 \psi \right]^{-1} \left[D_{xy}^2 f\right] \left[D^2 \phi \right]^{-1} y^*\right)^{\top} + x^{*, \top} \left[D^2 \psi \right]^{-1} = 0 \\
  \label{ap:eqn:deriv-zero2}
  & \left[D_y g\right] \left[D^2 \phi \right]^{-1} + x^{*,\top} \left[D^2 \psi \right]^{-1} \left[D_{xy}^2 g\right] \left[D^2 \phi \right]^{-1} + y^{*, \top} \left[D^2 \phi \right]^{-1} = 0
\end{align}

We plug these equation into each other, to obtain 

\tiny
\begin{align*}
  &\left[D_x f\right] \left[D^2 \psi \right]^{-1} 
  - \left(\left[D_y g\right]  
  + x^{*,\top} \left[D^2 \psi \right]^{-1} \left[D_{xy}^2 g\right] \right)\left[D^2 \phi \right]^{-1} \left[D_{yx}^2f\right] \left[D^2\psi\right]^{-1}
  + x^{*, \top} \left[D^2 \psi \right]^{-1} = 0 \\
  &  \left[D_y g\right] \left[D^2 \phi \right]^{-1} 
  -  \left(\left[D_x f\right] + y^{*,\top} \left[D^2 \phi \right]^{-1} \left[D_{yx}^2 f\right] \right)\left[D^2 \psi \right]^{-1} \left[D_{xy}^2 g\right] \left[D^2 \phi \right]^{-1} + y^{*, \top} \left[D^2 \phi \right]^{-1} = 0
\end{align*}
\normalsize
Solving the above for $x^*$ and $y^*$, we obtain 
\tiny
\begin{align*}
  & x^* = -\left(\left[D^2\psi\right]^{-1} - \left[D^2\psi\right]^{-1} \left[D_{xy}^2 f\right] \left[D^2 \phi \right]^{-1} \left[D_{yx}^2 g\right] \left[D^2\psi\right]^{-1}  \right)^{-1} \left( \left[D^2\psi\right]^{-1} \left[D_x f \right]^{\top} - \left[D^2 \psi\right]^{-1} \left[D_{xy}^2 f \right] \left[D^2 \phi\right]^{-1} \left[D_y g\right]^{\top}  \right) \\
  & y^* = -\left(\left[D^2\phi\right]^{-1} - \left[D^2\phi\right]^{-1} \left[D_{yx}^2 g\right] \left[D^2 \psi \right]^{-1} \left[D_{xy}^2 f\right] \left[D^2\phi\right]^{-1}  \right)^{-1} \left( \left[D^2\phi\right]^{-1} \left[D_x g \right]^{\top} - \left[D^2 \phi\right]^{-1} \left[D_{yx}^2 g \right] \left[D^2 \psi\right]^{-1} \left[D_x f\right]^{\top}  \right).
\end{align*}
\normalsize
Once $x^{*}$ and $y^{*}$ have been computed, we can update the dual variables as $x_{k + 1} = x_{k} + x^{*}$ and $y_{k + 1} = y_{k} + y^{*}$.

\paragraph{Computing the updates in practice:} Just like competitive gradient descent \citep{schafer2019competitive}, CMD requires the solution of a system of linear equations at each step. 
While this may seem prohibitively expensive at first, \citep{schafer2019competitive,schafer2019implicit} show that CGD can be scaled to problems with millions of degrees of freedom.
This is achieved by using Krylov subspace methods such as the conjugate gradient or GMRES algorithms \citep{saad2003iterative} combined with mixed-mode automatic differentiation that allows to compute Hessian-vector products almost as cheaply as gradients \citep{pearlmutter1994fast}.
By using $\left[D^2 \psi \right]$ and $\left[D^2 \phi \right]$ as a preconditioner, the matrices that have to be inverted at each step are perturbations of the identity
\begin{align}
  \label{ap:eqn:matrix-1}
  \left(\Id - \left[D^2\psi\right]^{-\frac{1}{2}} \left[D_{xy}^2 f\right] \left[D^2 \phi \right]^{-1} \left[D_{yx}^2 g\right] \left[D^2\psi\right]^{-\frac{1}{2}}  \right) \\
  \label{ap:eqn:matrix-2}
  \left(\Id - \left[D^2\phi\right]^{-\frac{1}{2}} \left[D_{yx}^2 g\right] \left[D^2 \psi \right]^{-1} \left[D_{xy}^2 f\right] \left[D^2\phi\right]^{-\frac{1}{2}}  \right).
\end{align}
As discusssed in \citep{schafer2019competitive}, competing methods become unstable if the perturbations is large. 
If the perturbation is small, conjugate gradient or GMRES algorithms converge quickly, resulting in minimal overhead. 
\citep{schafer2019competitive} show that this adaptivity, together with the fact that CGD can use larger step sizes, allows it to outperform competing methods even when fairly accounting for the cost of the matrix inversion.
The computational cost can be reduced further by using the solution of the linear system in the last iteration as a warm start for the solution in the present iteration.
Finally, we can see from Equations~\ref{ap:eqn:deriv-zero1}~and~\ref{ap:eqn:deriv-zero2} that once $x^{*}$ ($y^{*}$) has been computed, $y^{*}$ ($x^{8}$) can be computed by only inverting $\left[D^2 \phi\right]$ ($\left[D^2 \psi\right]$), which can often be done in closed form.
Therefore we can alternatingly invert the matrices in Equations~\ref{ap:eqn:matrix-1}~and~\ref{ap:eqn:matrix-2}, one at each iteration.

\paragraph{Numerical experiments:} We will now provide numerical evidence for the practical performance of CMD. 
The Julia-code for the numerical experiments presented here can be found under \url{https://github.com/f-t-s/CMD}.
As discussed in Section 3 of the paper, naively combining cgd with a projection step can result in convergence to spurious stable points even in convex problems, due to the \emph{empty threats phenomenon}. 
The same argument applies to all other methods described in \citep{schafer2019competitive}[Section 3] as incuding a \emph{``competitive term''}, with the exception of OGDA and the closely related extragradient method.
We believe that hte empty threats phenomenon rules out projected versions of algorithms affected by it and therefore focus our numerical comparison on the projected extragradient method (PX) of \citep{korpelevich1977extragradient}.
We also focus on the special case of CMW in this section, leaving a more thorough exploration of other constraint sets and Bregman potentials to future work.

A first benefit of CMD is that it allows us to extend the robustness properties of CGD to conically constrained problems. 
As discussed in \citep{schafer2019implicit}, CGD is robust to strong interactions , without adjusting its step size.

To showcase this property, we consider the simple bilinear zero-sum game $f(x,y) = \alpha (x - 0.1)(y - 0.1) = -g(x, y)$ with $x$ and $y$ constrained to lie in $\Reals_+$. 
While the projected extragradient method converges faster for small $\alpha$ than CMD, we observes that the latter converges over the entire range of values for $\alpha$, whereas the extragradient method diverges as $\alpha$ gets too large (see Figure~\ref{fig:comparison}).

\begin{figure}
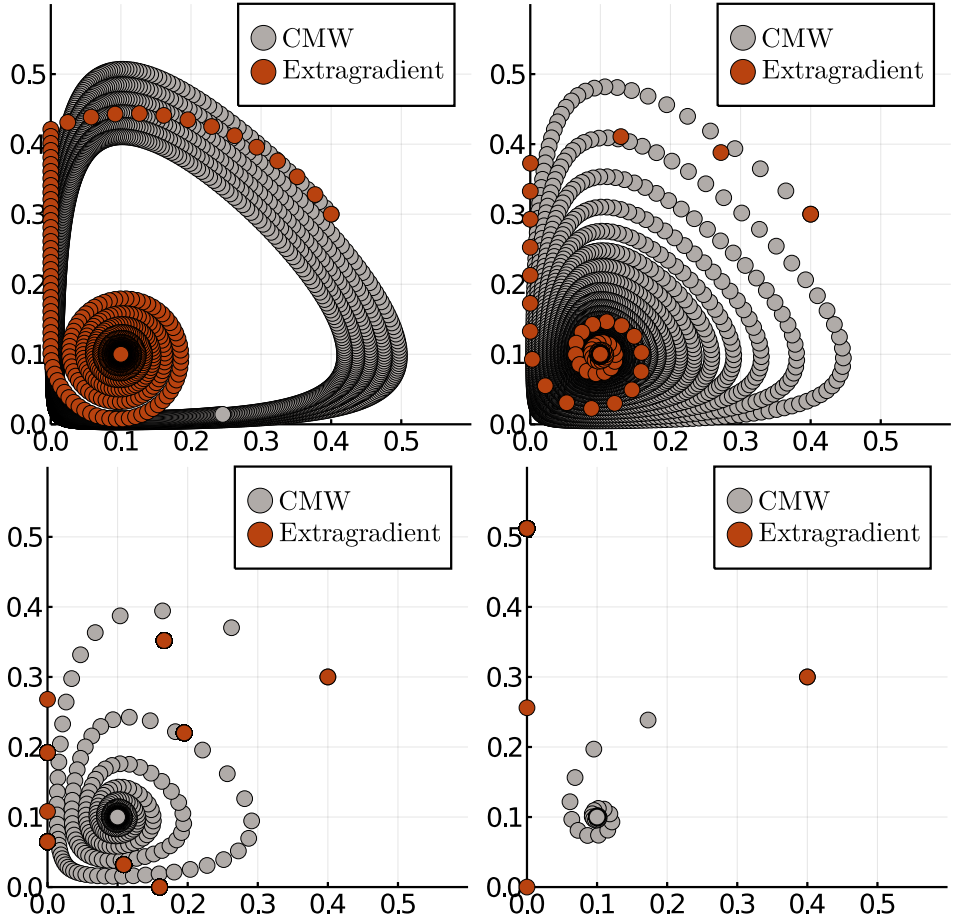

  \centering
  \includegraphics[width=0.45\textwidth]{figures/bilinear_positive_3to0.pdf}
  \includegraphics[width=0.45\textwidth]{figures/bilinear_positive_3to1.pdf}
  \includegraphics[width=0.45\textwidth]{figures/bilinear_positive_3to2.pdf}
  \includegraphics[width=0.45\textwidth]{figures/bilinear_positive_3to3.pdf}
  \caption{\label{ap:fig:comparison} CMW and PX applied to $f(x,y) = \alpha (x-0.1)(y-0.1) = -g(x,y)$ ($\alpha \in \{0.1, 0.3, 0.9, 2.7\}$). For small $\alpha$, PX converges faster but for large $\alpha$ it diverges.}
\end{figure}

We now move to a high-dimensional example that will further showcase the advantages of combining CGD with mirror descent.
We consider the robust regression problem on the probability simplex given by 
\begin{equation}
  \label{ap:eqn:regression-objective}
  \min \limits_{x:\ x \geq 0,\ \mathbf{1}^{\top} x = 1} \|Ax - b\|^2
\end{equation}
for $A \in \Reals^{50 \times 5000}$, $A_{ij} \sim \mathcal{N}\left(0, \Id\right)$ i. i. d., $\epsilon \sim \mathcal{N}\left(0, \Id\right)$, and $b = (A_{:,1} + A_{:,2})/2 + \epsilon$.
In order to enforce the normalization constraint $\mathbf{1}^{\top} x = 1$, we introduce a Lagrange multiplier $y \in \Reals$ and solve the competitive optimization problem given as
\begin{equation}
  \min \limits_{x \in \Reals_{+}^{5000}} \|Ax - b\|^2 + y (\mathbf{1}^{\top} x - 1), \quad \min \limits_{y \in \Reals} - y (\mathbf{1}^{\top} x - 1).
\end{equation}
We use different inverse step sizes $\alpha \in \{100, 1000\}$ for $x$ and $\beta \in \{1, 10, 100, 1000\}$ for $y$ and solve the competitive optimization problem using CMW and PX.
We then plot the loss incurred when using $x / \mathbf{1}^{\top} x$ as a function of the number of iterations.
The extragradient method stalled for all step sizes that we tried, which is why we introduce extramirror (PXM), a mirror descent version of extragradient that at each step uses the gradient computed in the next iteration of mirror descent to perform a mirror descent update in the present iteration.
As shown in Figure~\ref{ap:fig:iterplots} CMW is the only algorithm that converges over the entire range of $\alpha$ and $\beta$. 
The projected extragradient method always stalls, while the extramirror algorithm diverges and produces NAN values for the largest step size. 
Generally speaking, CMW converges faster for larger step sizes, while PXM converges faster for smaller step sizes, as shown in Figure~\ref{ap:fig:iterplots}
\begin{figure}
  \centering
  \includegraphics[width=0.45\textwidth]{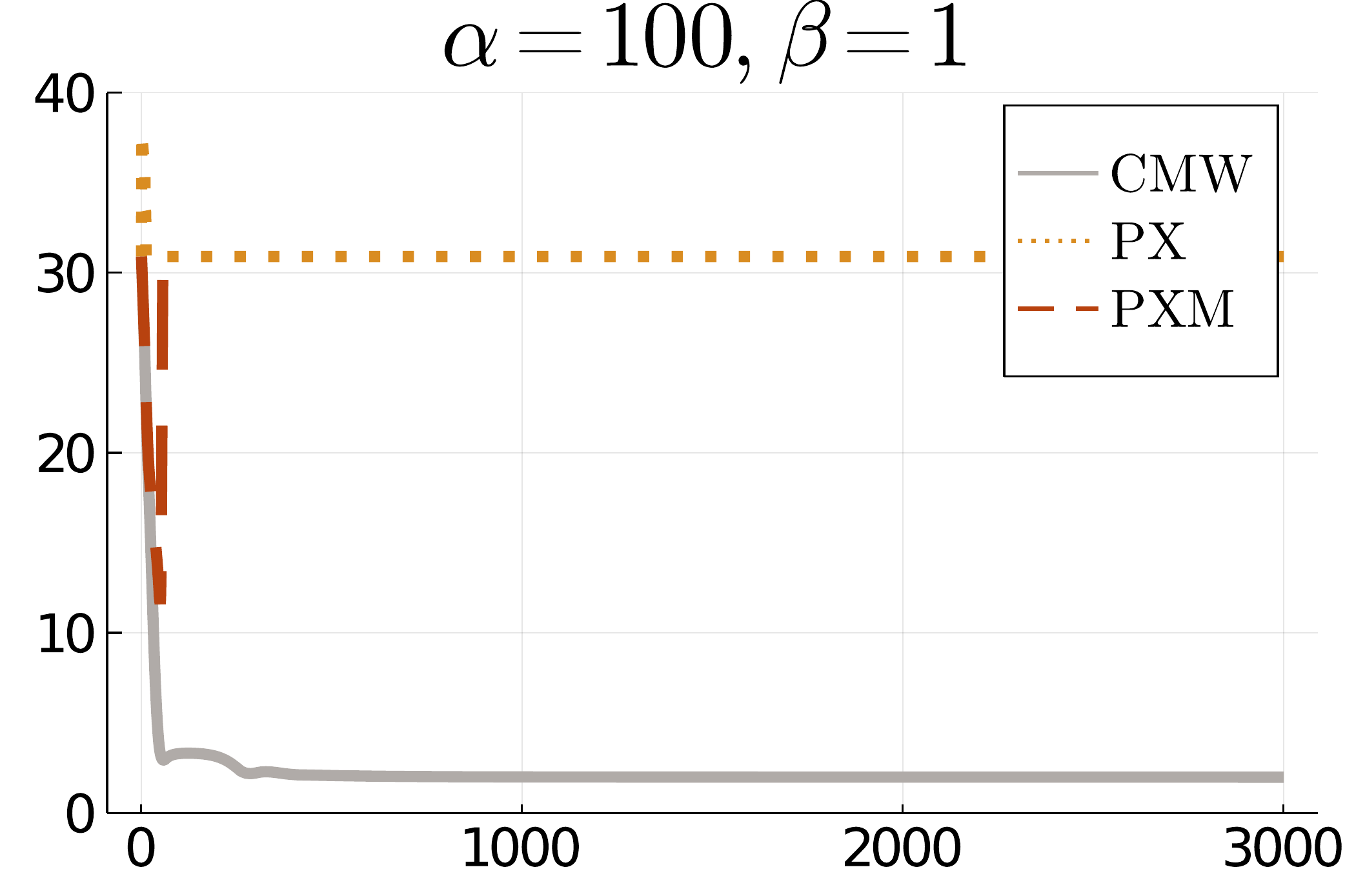}
  \includegraphics[width=0.45\textwidth]{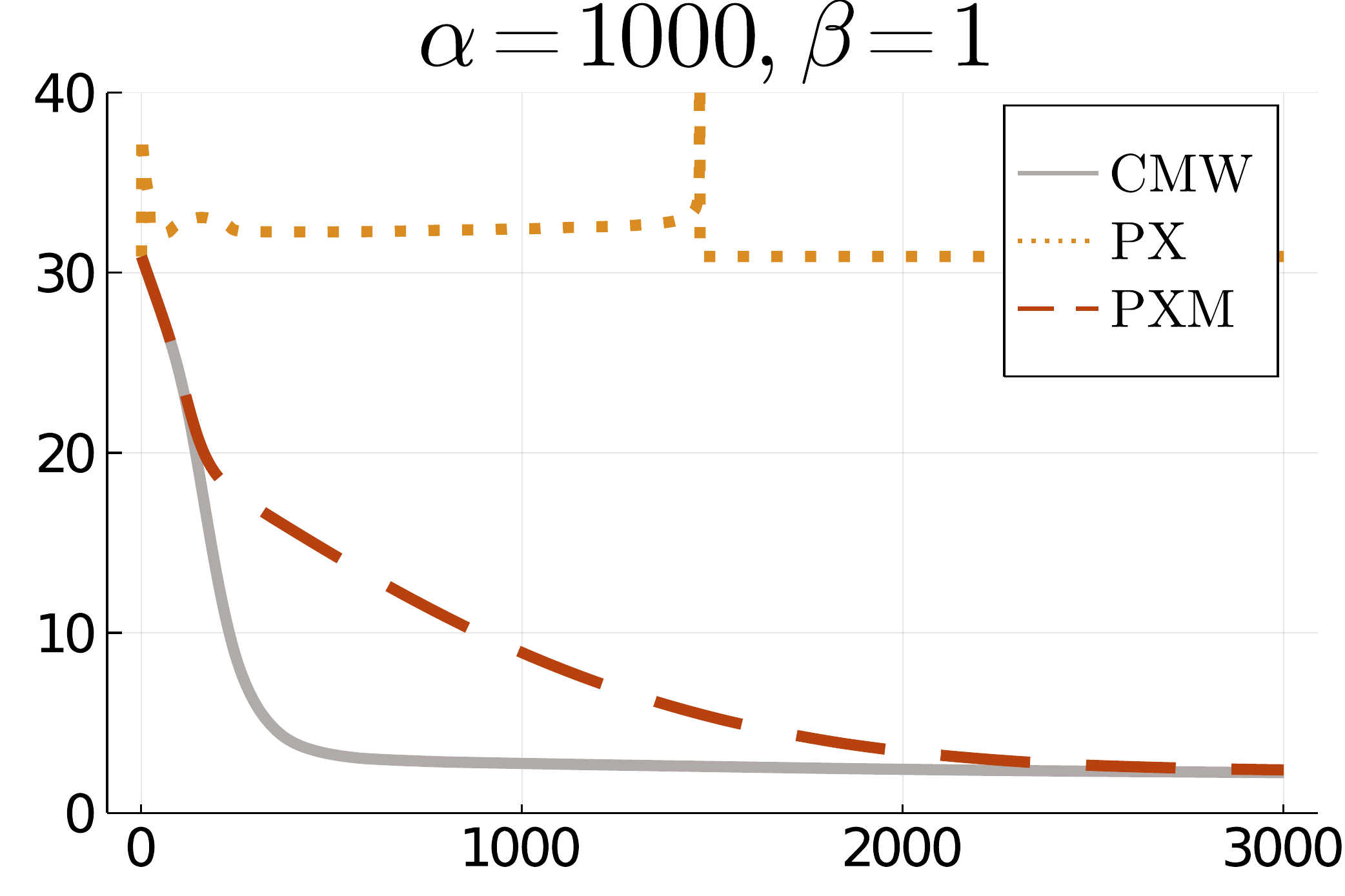}
  \includegraphics[width=0.45\textwidth]{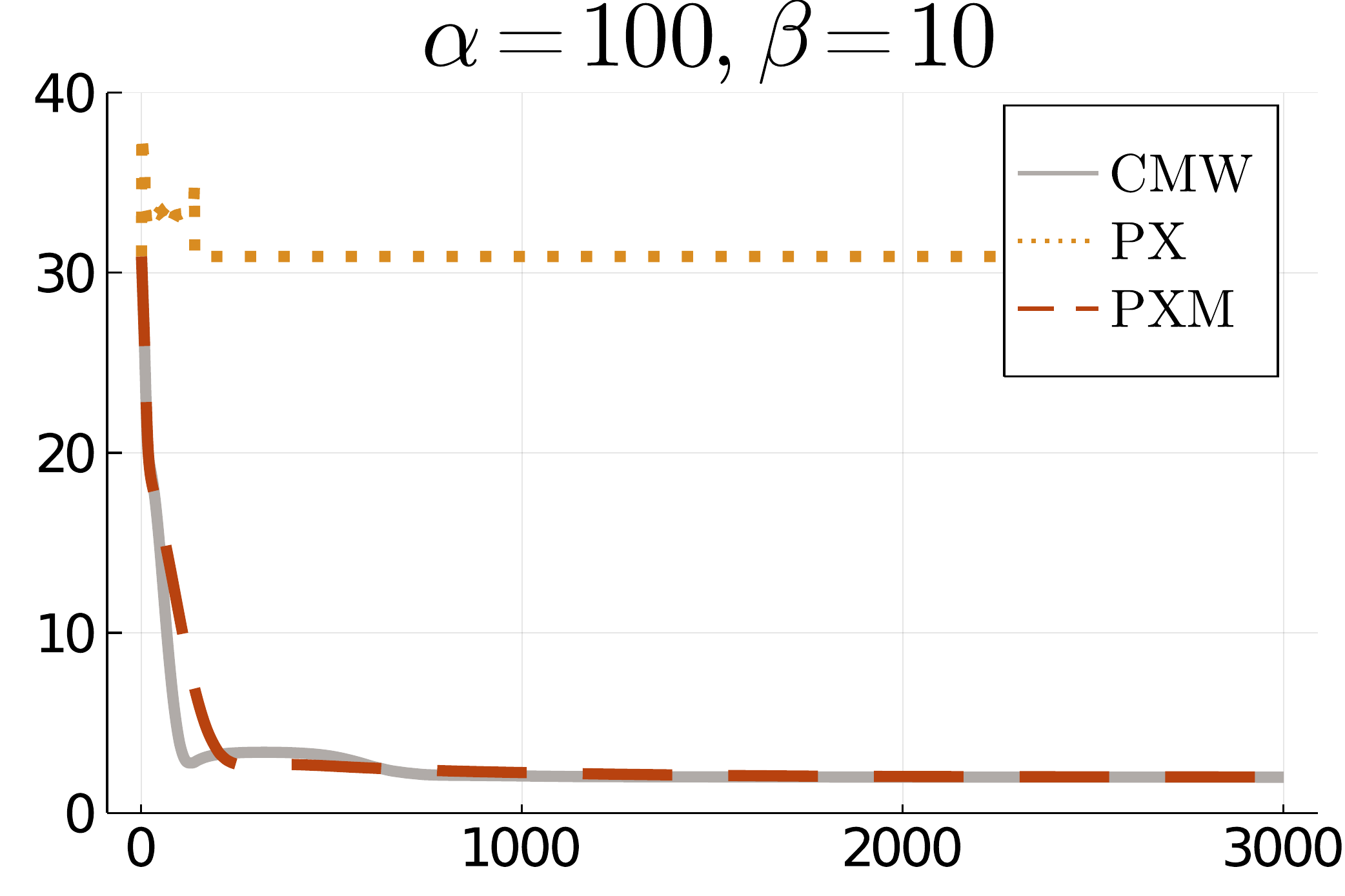}
  \includegraphics[width=0.45\textwidth]{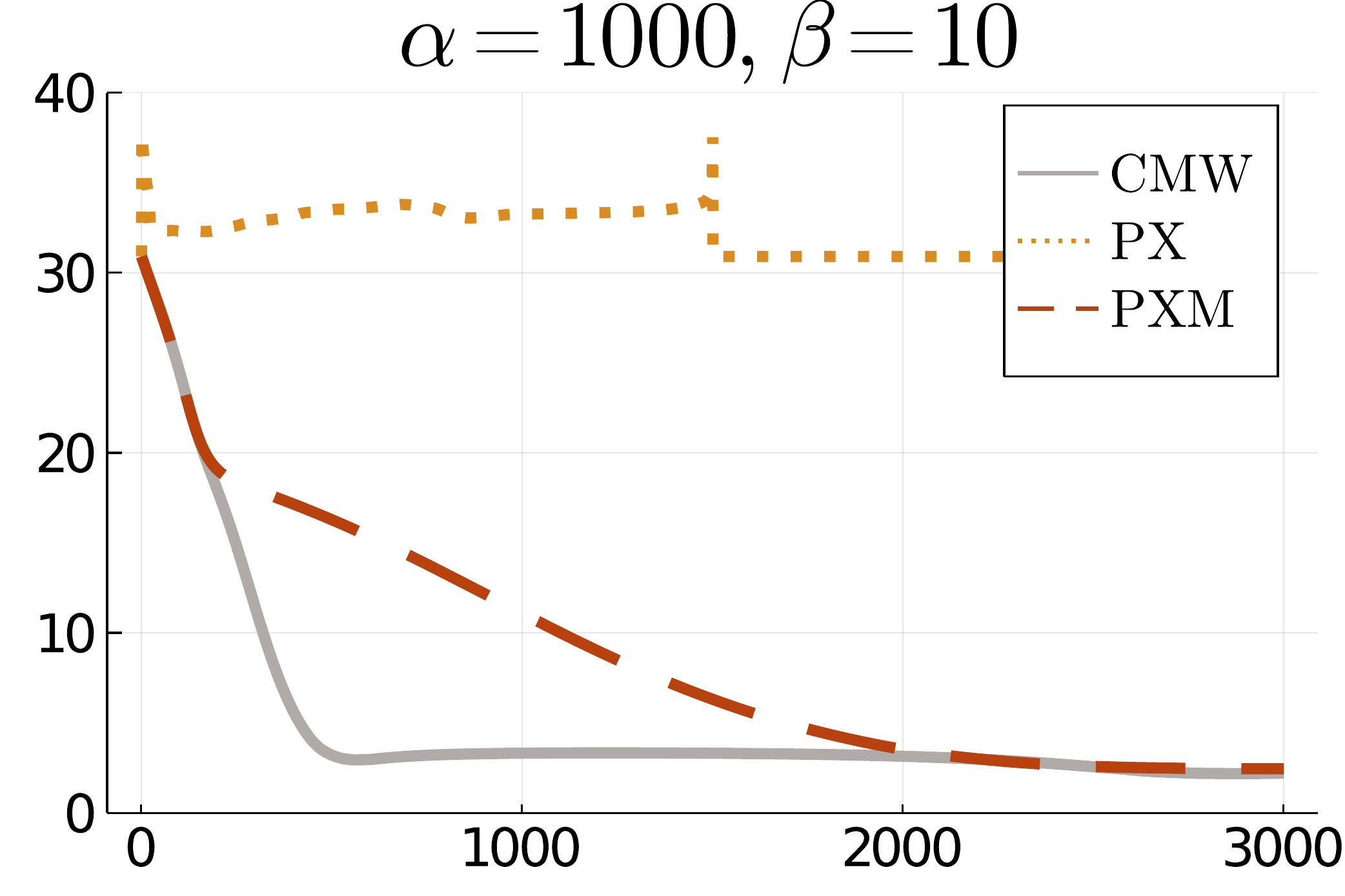}
  \includegraphics[width=0.45\textwidth]{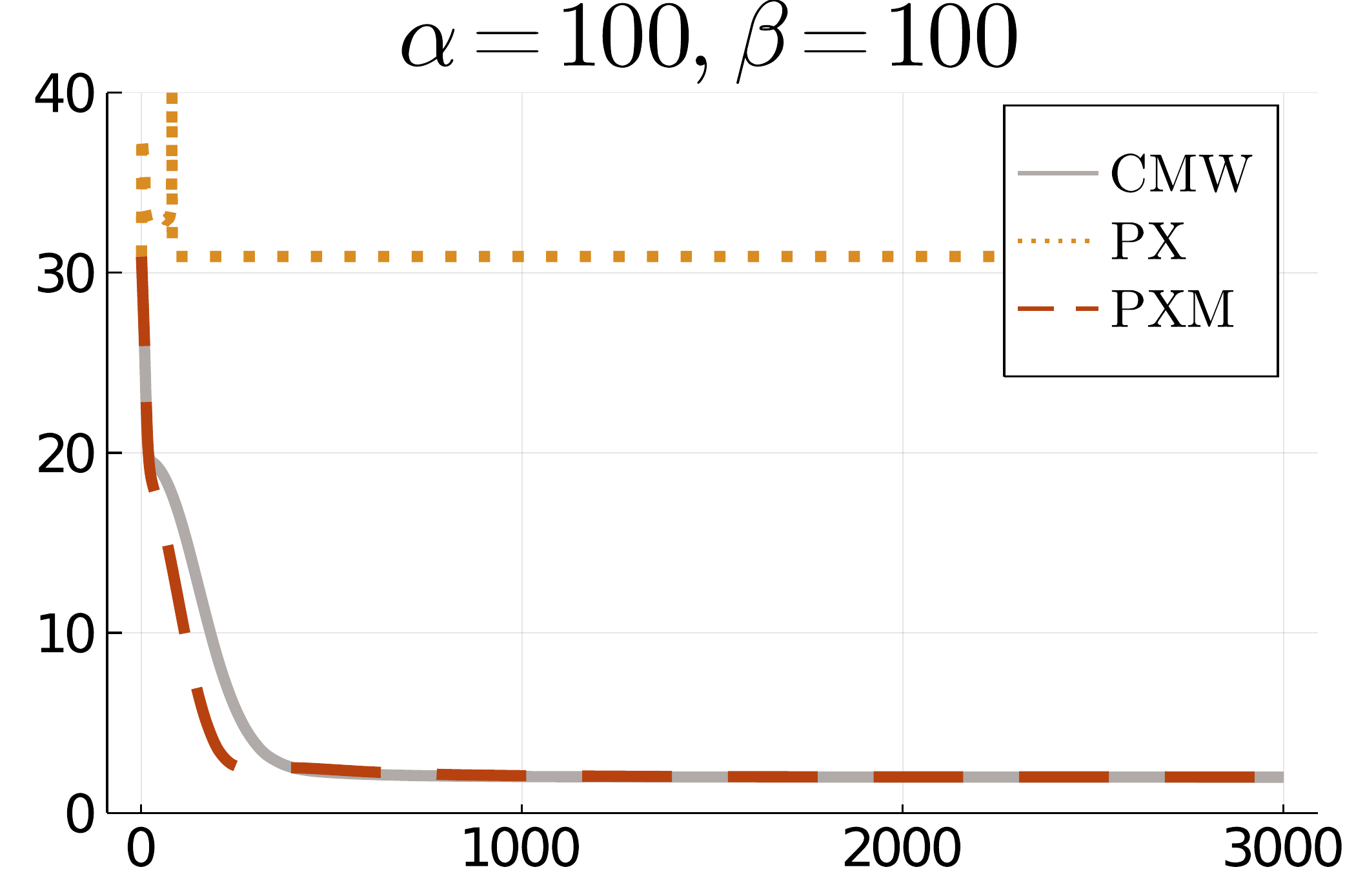}
  \includegraphics[width=0.45\textwidth]{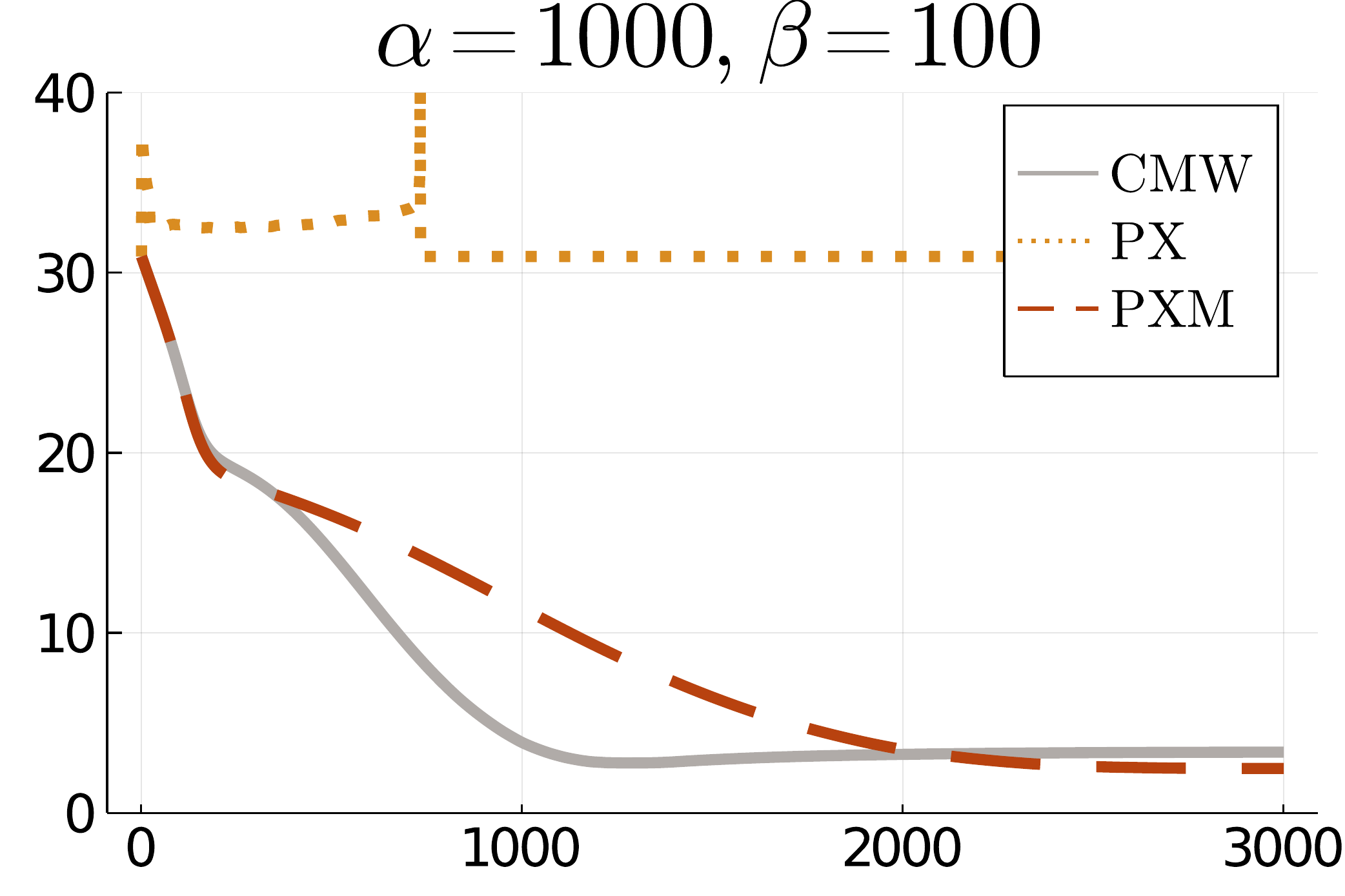}
  \includegraphics[width=0.45\textwidth]{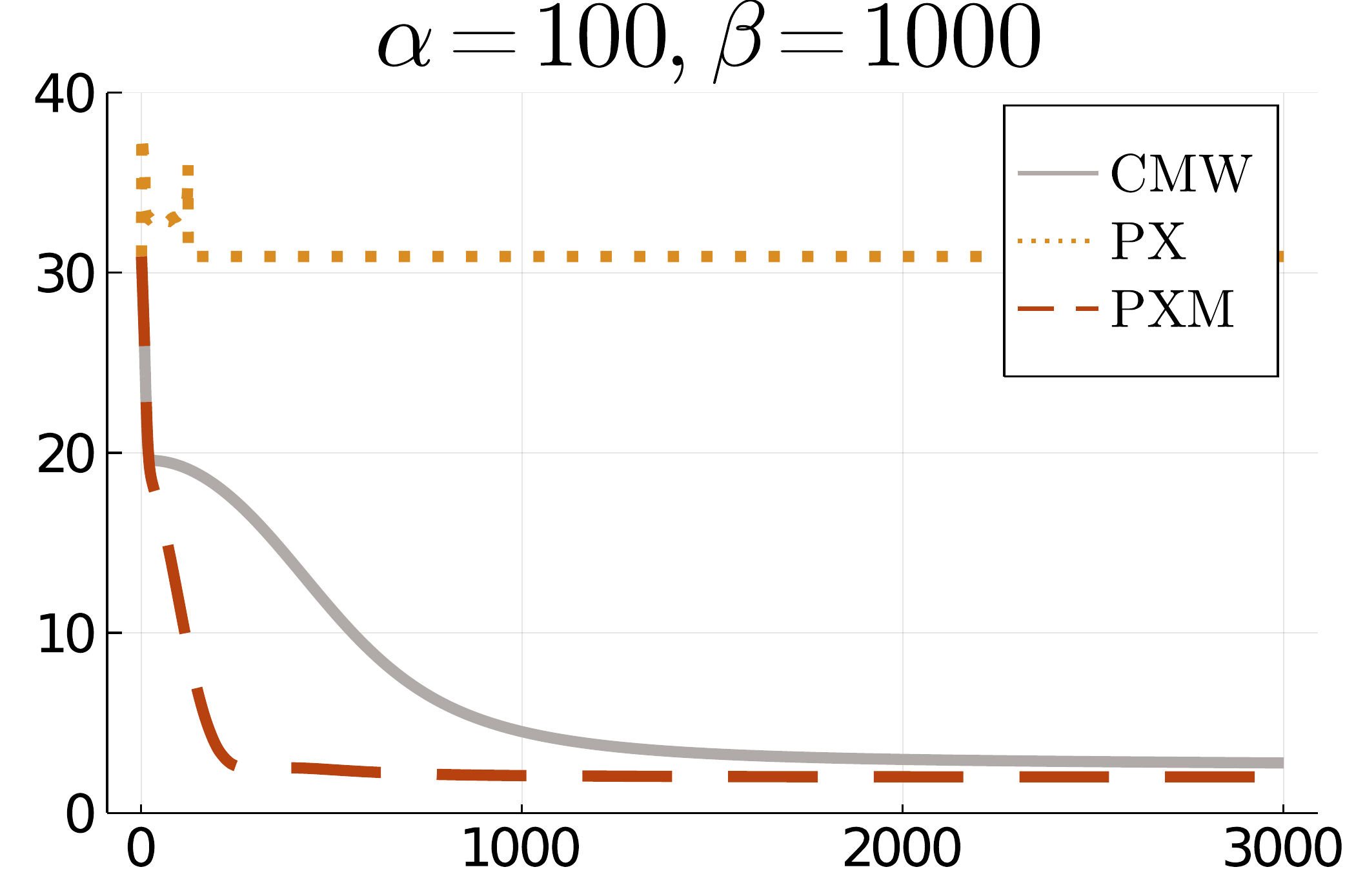}
  \includegraphics[width=0.45\textwidth]{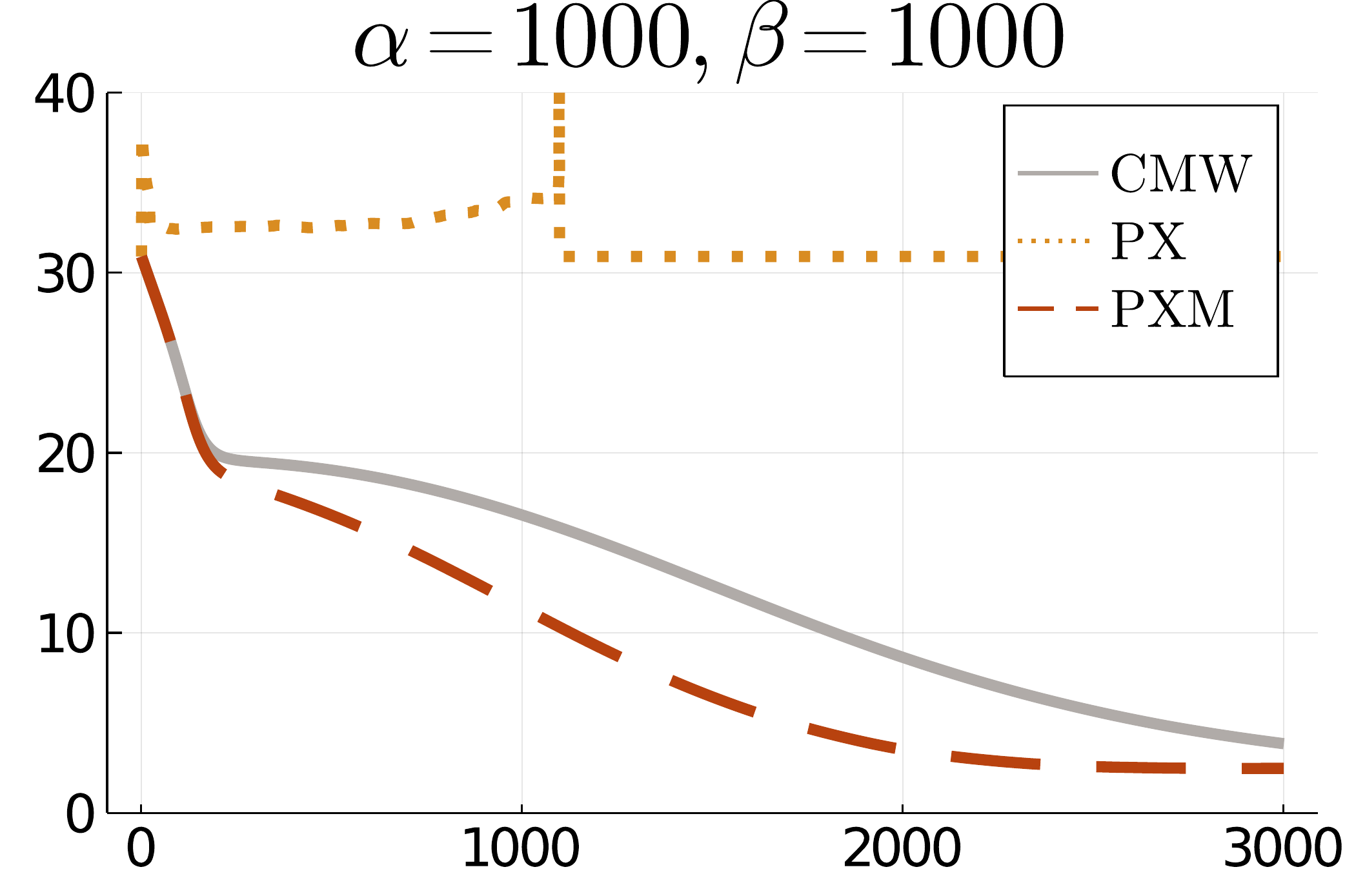}
  \caption{\label{ap:fig:iterplots} We plot the objective value in Equation~\ref{ap:eqn:regression-objective} (after normalization of $x$) compared to outer iterations. In the first panel, PXM diverges and produces NAN values, which is why the plot is incomplete}
\end{figure}
Of course, to compare apples to apples we need to account for the complexity of solving the matrix inverse in CMW.
To this end, we show in Figure~\ref{ap:fig:gradplots} the objectuve value compared to the number of gradient computations and Hessian-vector products.
Therefore, each outer iteration of the extragradient and extramirror methods amounts to an $x$-axis value of $4$, while an outer iteration of CMW, where the conjugate gradient solver requires $k$ steps, amounts to an $x$-axis value of $4 + 2k$.

\begin{figure}
  \centering
  \includegraphics[width=0.45\textwidth]{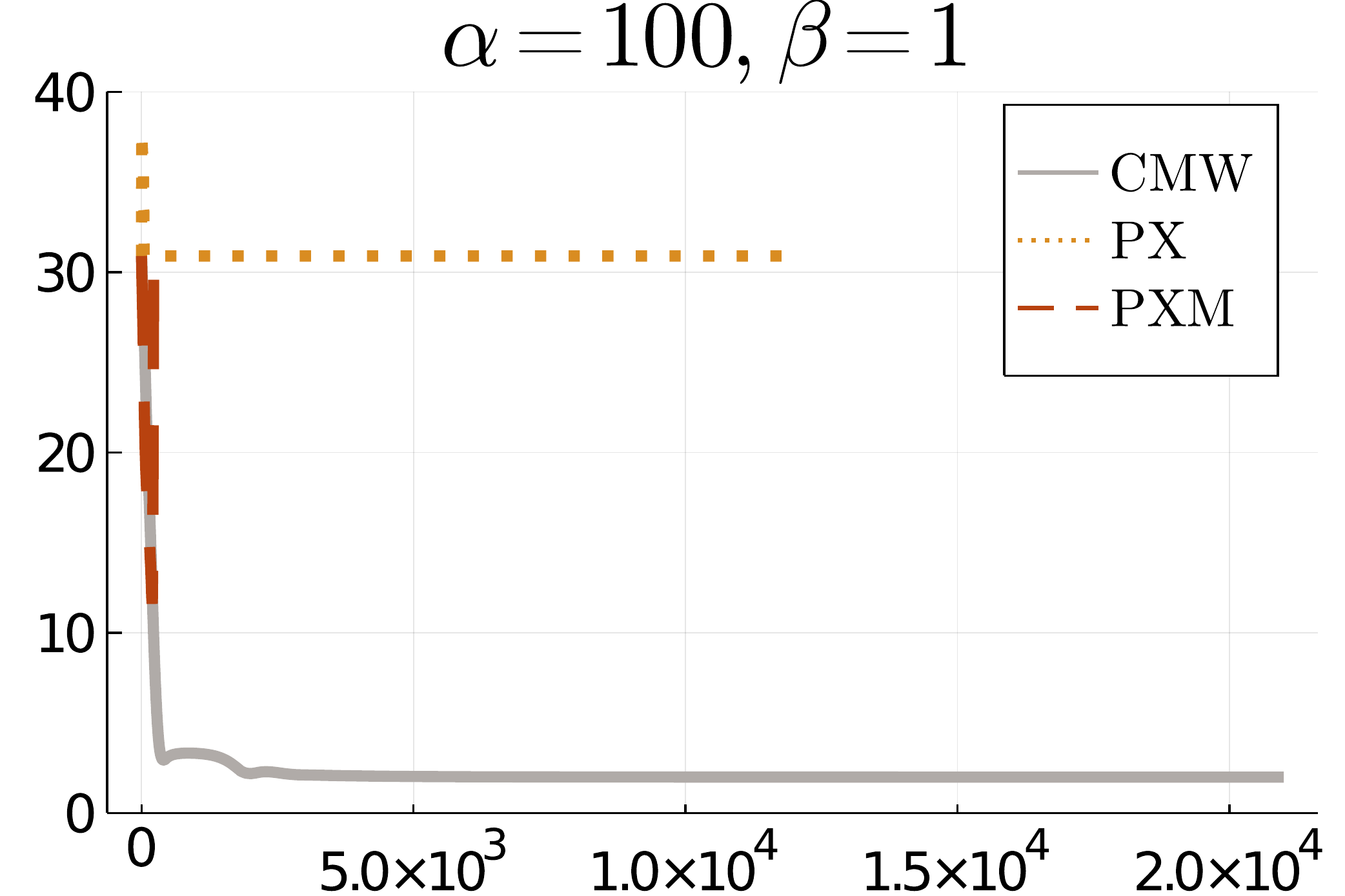}
  \includegraphics[width=0.45\textwidth]{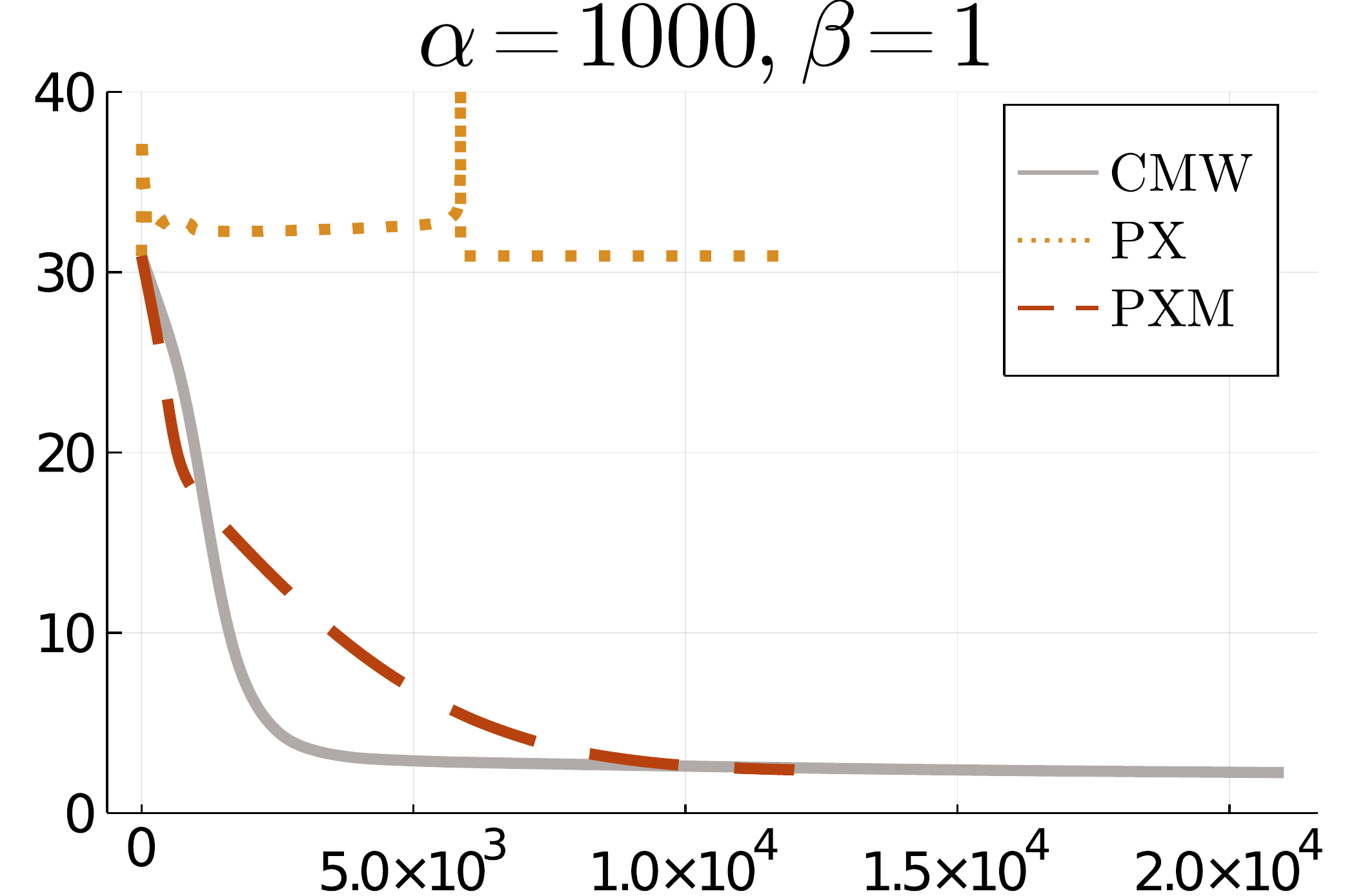}
  \includegraphics[width=0.45\textwidth]{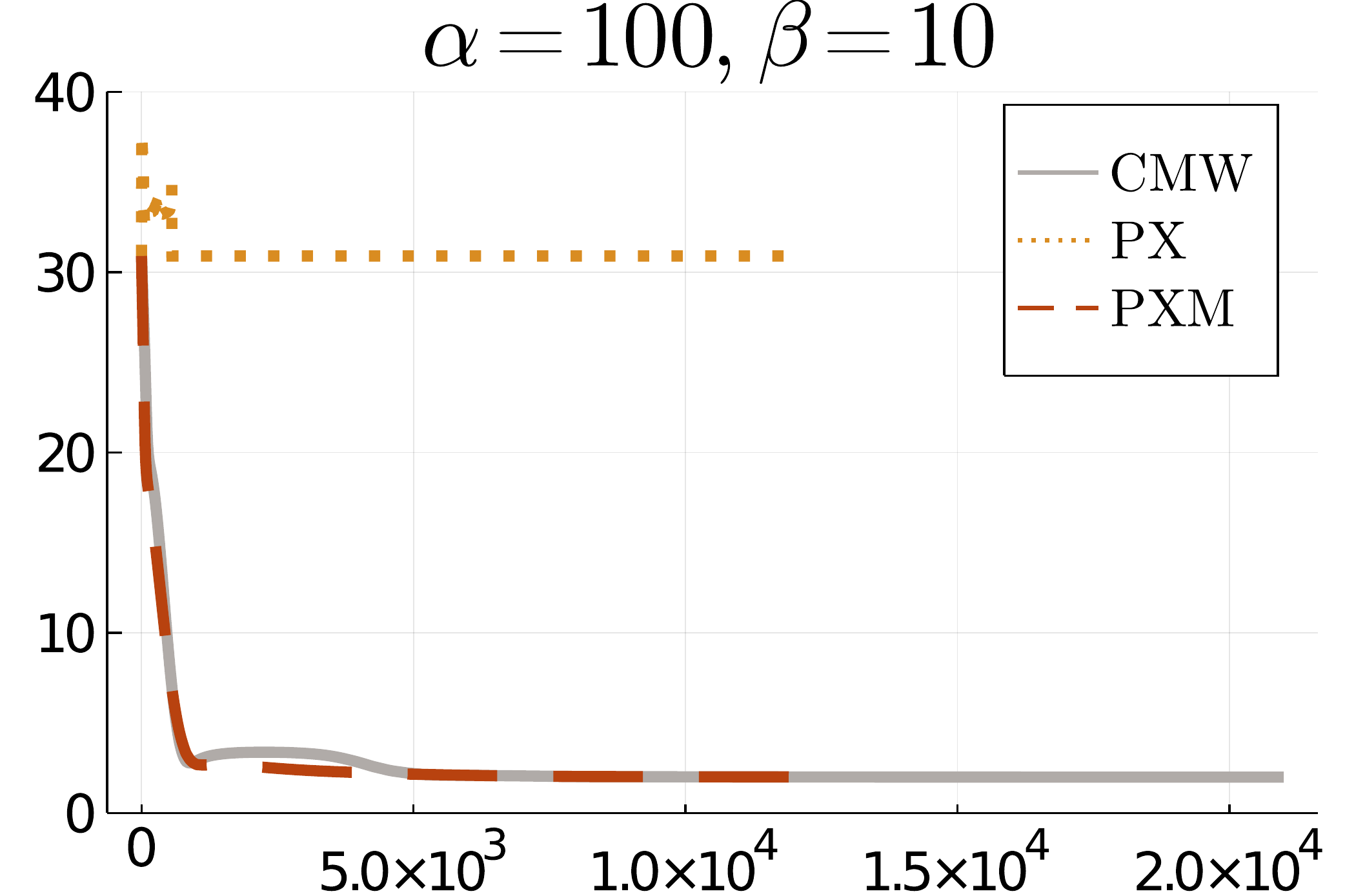}
  \includegraphics[width=0.45\textwidth]{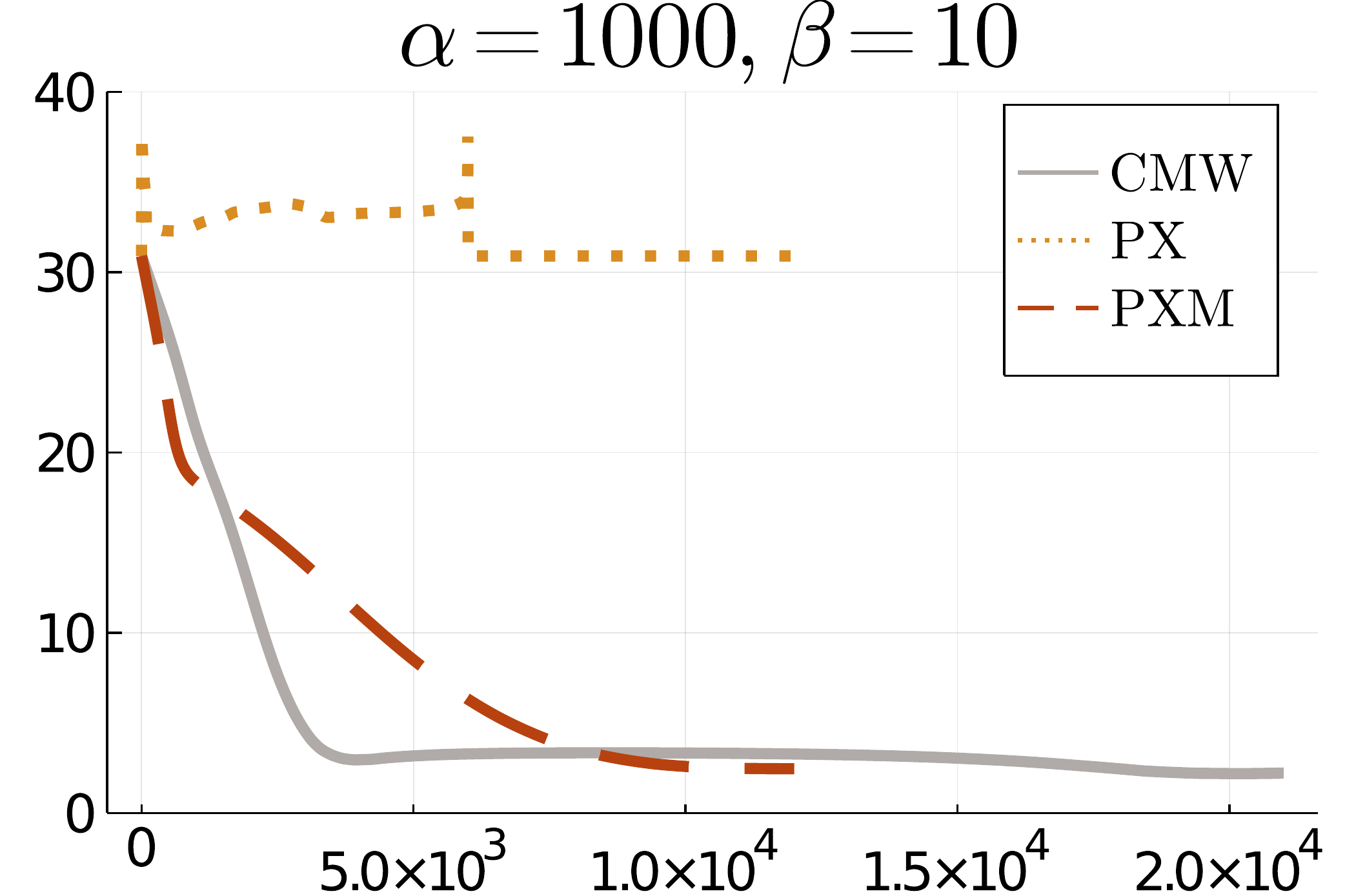}
  \includegraphics[width=0.45\textwidth]{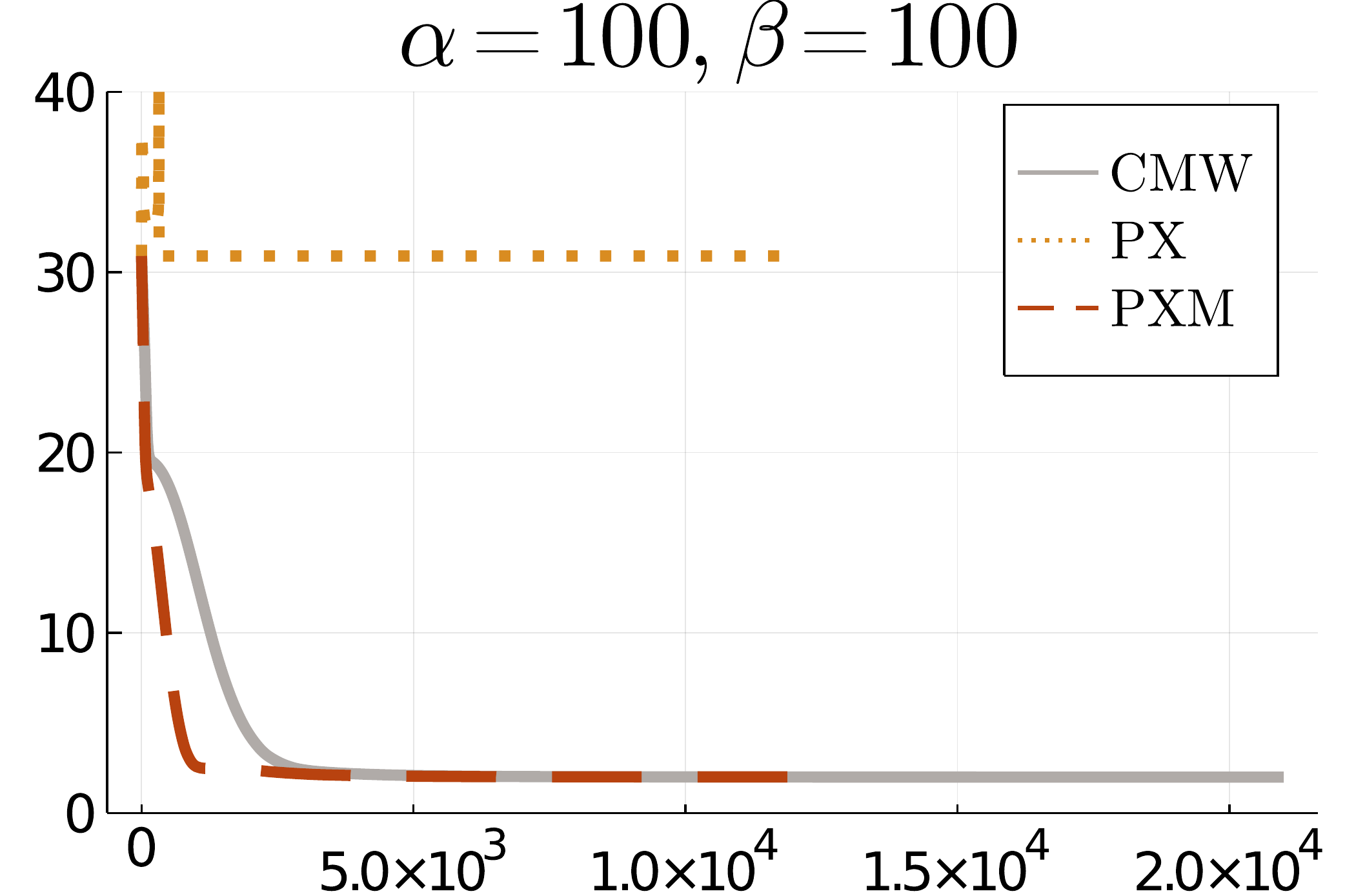}
  \includegraphics[width=0.45\textwidth]{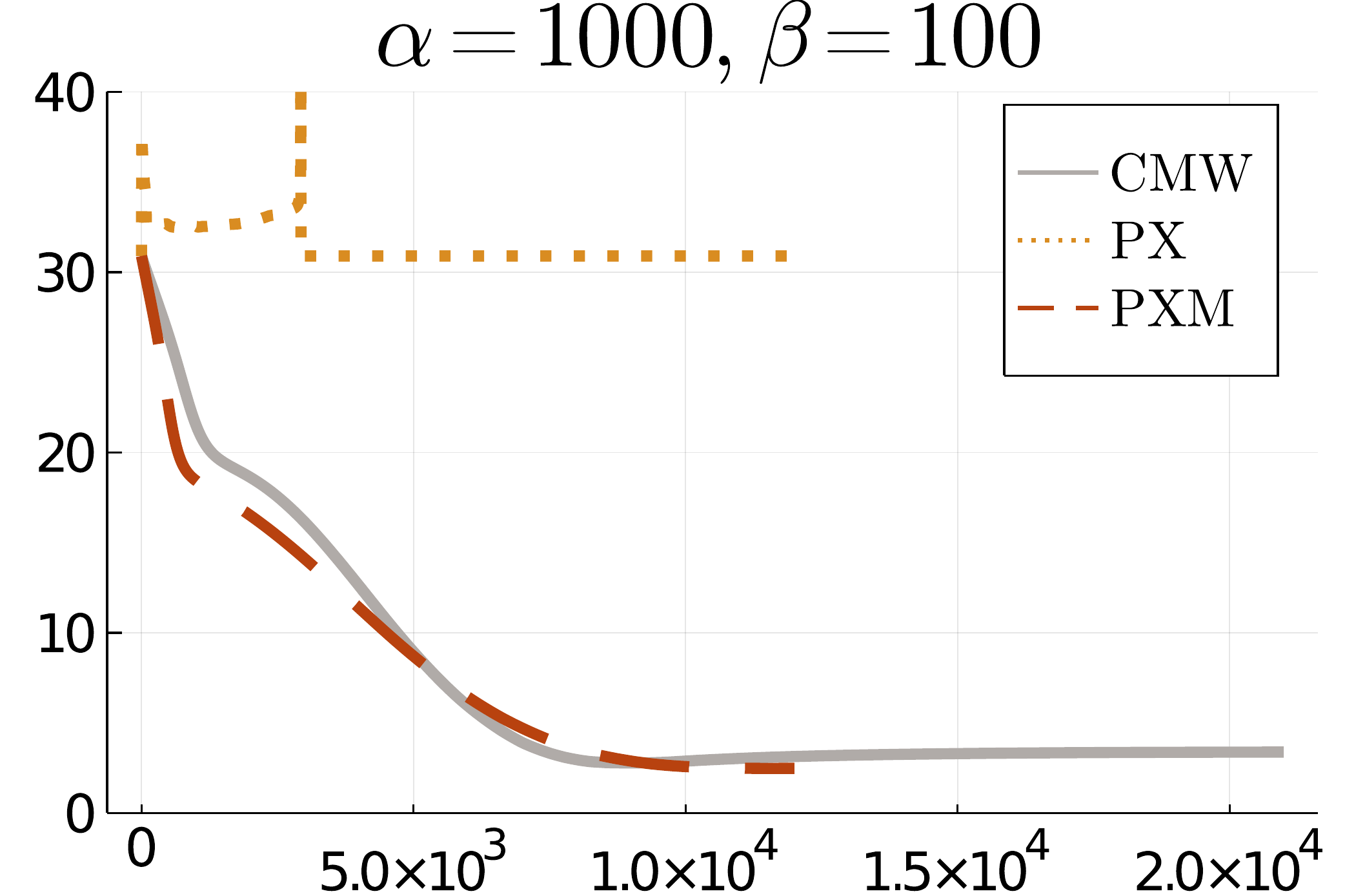}
  \includegraphics[width=0.45\textwidth]{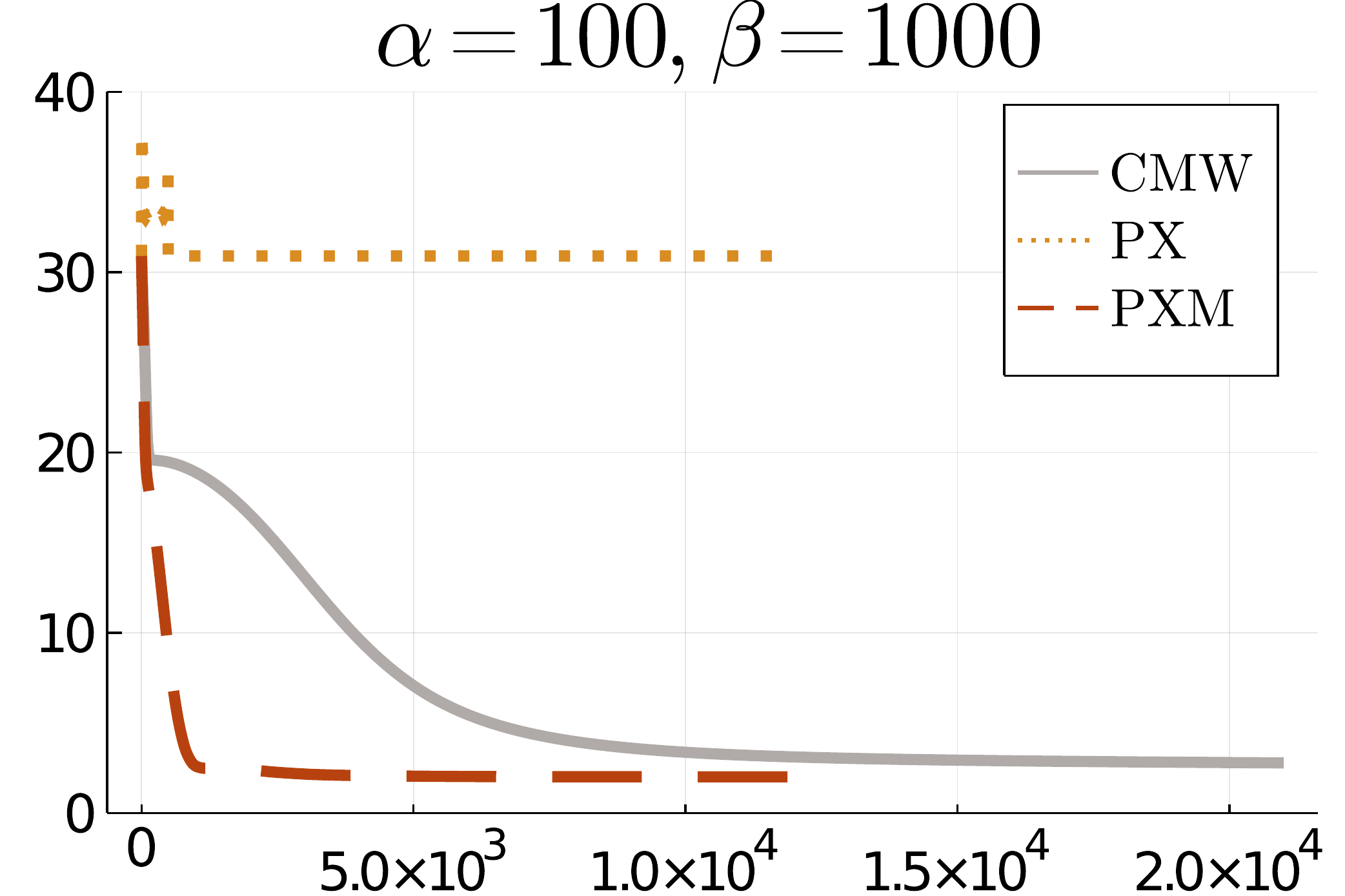}
  \includegraphics[width=0.45\textwidth]{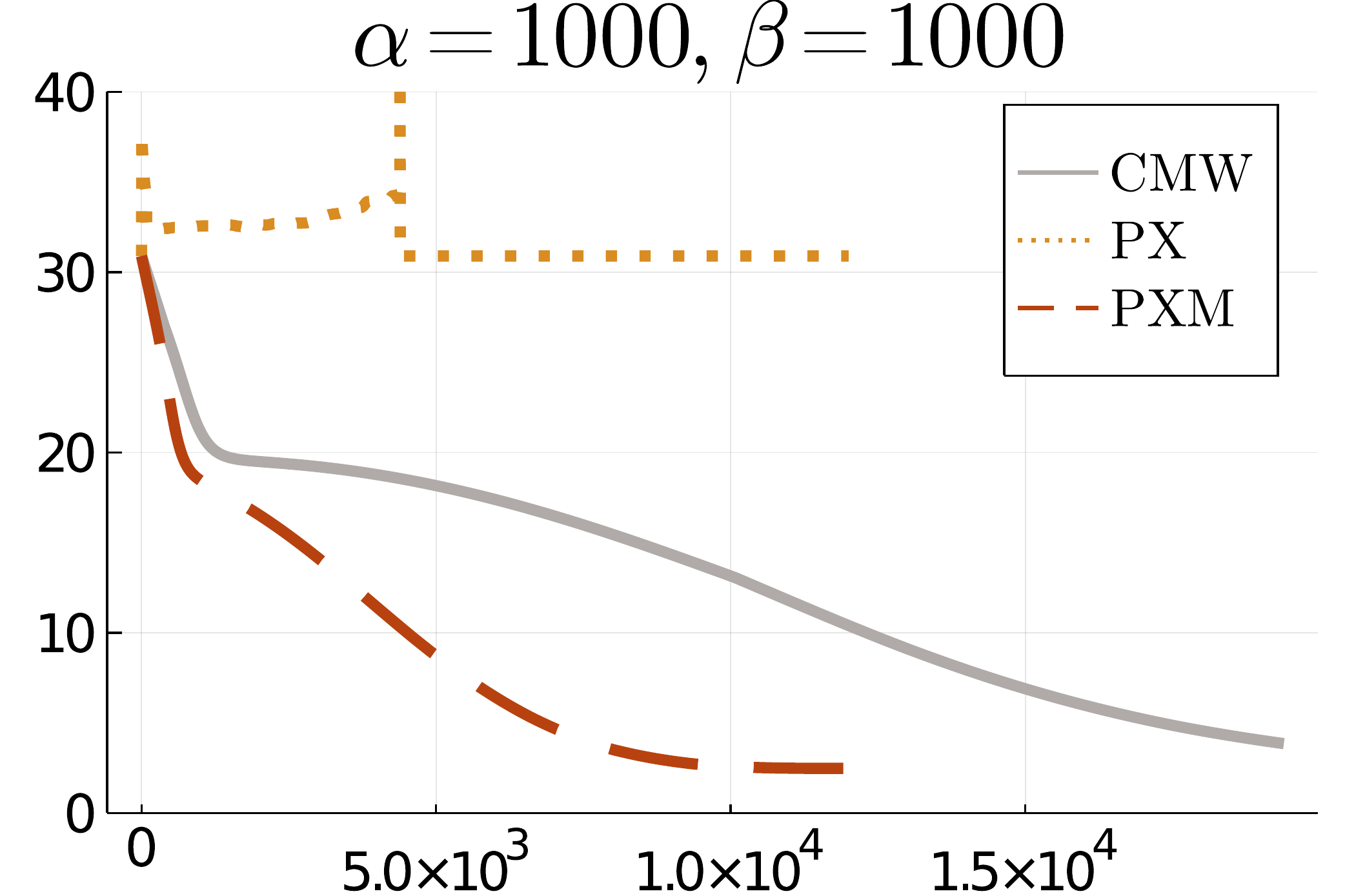}
  \caption{\label{ap:fig:gradplots} We plot the objective value in Equation~\ref{ap:eqn:regression-objective} (after normalization of $x$) compared to the number of gradient computations and Hessian-vector products, accounting for the inner loop of CMW.}
\end{figure}

Similar to \citep{schafer2019competitive}, we observe that even when fairly accounting for the complexity of the matrix inverse, CMW is competitive with PXM, beating it for larger step sizes, while being a bit slower for smaller step sizes.
Importantly, it is significantly more robust and converges even in settings where the competing methods diverge.

\end{document}

%% file: figures/tikz/pcgd_failure.tex
\draw[->, black] (-1,0) -- (3,0);
\draw[->, black] (0,-1) -- (0,3);
\node at (3.0,-0.25) {$x$};
\node at (-0.25,3.0) {$y$};
\node at (-0.5,-0.5) {$(0,0)$};

\node at (-0.5, 1.0) {$(0,\frac{2}{3})$}; 
\filldraw (2.0, 2.0) circle (1pt);
\draw[->, black] (2.0 - 0.24, 2.0 + 0.08) -- (0.8 + 0.24, 2.4 - 0.08);
\filldraw (0.8, 2.4) circle (1pt);
\draw[->, black] (0.8 - 0.16, 2.4 - 0.064) -- (0.0 + 0.16, 2.08 + 0.064);
\filldraw (0.0, 2.08) circle (1pt);
\draw[->, black] (0, 2.08 - 0.0916) -- (0.0, 1.622 + 0.0916);
\filldraw (0.0, 1.622) circle (1pt);
\filldraw (0.0, 1.4528) circle (1pt);
\filldraw (0.0, 1.35245) circle (1pt);
\filldraw (0.0, 1.34098) circle (1pt);
\filldraw (0.0, 1.33638) circle (1pt);

\node at (1.1, 3.5) {\large $\min \limits_{y} {\color{silver}-2 x y}  + {\color{rust}(y - 1)^2}$};

\draw[->, silver, very thick, dashed] (-0.25, 4/3) -- (-1.0, 4/3); 

\draw[->, silver, very thick, dashed] (0.25, 4/3 - 0.25) -- (0.25, 4/3 - 1.25); 
\draw[->, rust, very thick] (-0.25, 4/3 + 0.25) -- (-0.25, 4/3 + 1.25); 